\documentclass[a4paper]{gtart} 
\usepackage{amsmath, amssymb, latexsym, amscd, graphicx}
\usepackage{psfrag}
\usepackage{subfigure}
\usepackage{footnote}

\usepackage[inner=25mm, outer=25mm, textheight=225mm]{geometry}

\usepackage[small,bf]{caption}
\def\tri{\mathcal{T}}

\def\bkR{{\rm I\kern-.17em R}}
\def\R{\bkR}
\def\bkZ{{\rm Z\kern-.28em Z}}
\def\Z{\bkZ}
\def\prho{\overline{\rho}}

\DeclareMathOperator{\im}{im}
\DeclareMathOperator{\Sym}{Sym}
\DeclareMathOperator{\Alt}{Alt}

\theoremstyle{plain}
\newtheorem{theorem}{Theorem}
\newtheorem{lemma}[theorem]{Lemma}
\newtheorem{proposition}[theorem]{Proposition}
\newtheorem{corollary}[theorem]{Corollary}
\newtheorem{problem}[theorem]{Problem}

\theoremstyle{definition}

\newtheorem*{definition*}{Definition}

\newtheorem{example}[theorem]{Example}
\newtheorem{remark}[theorem]{Remark}

\def\tri{\mathcal{T}}
\def\di{\mathcal{D}}
\numberwithin{equation}{section}

\begin{document}

\title{Even triangulations of $\bf{n}$--dimensional pseudo-manifolds}
\author{J.\thinspace Hyam Rubinstein and Stephan Tillmann}

\begin{abstract}
This paper introduces even triangulations of $n$--dimensional pseudo-manifolds and links their combinatorics to the topology of the pseudo-manifolds. This is done via normal hypersurface theory and the study of certain symmetric representation. In dimension 3, necessary and sufficient conditions for the existence of even triangulations having one or two vertices are given. For Haken $n$--manifolds, an interesting connection between very short hierarchies and even triangulations is observed. 
\end{abstract}



\makeshorttitle


\section{Introduction}

Aitchison, Matsumoto and Rubinstein~\cite{AMR} discuss immersed hypersurfaces in cubed $n$--manifolds of non-positive curvature in terms of a holonomy representation into the symmetric group on $n$ letters. A necessary and sufficient condition to have such a representation is for each codimension--$2$ face of the cubing to have even order.
Joswig and Izmestiev~\cite{Joswig2002, IJ-2003} study $n$--dimensional polytopes and 
PL manifolds in terms of the symmetric group on $n+1$ letters, also obtaining representations if and only if all codimension--$2$ faces have even order. 

This paper synthesises these ideas into a theory for singular triangulations of $n$--dimensional pseudo-manifolds in which all codimension--$2$ faces have even order. From these \emph{even triangulations}, we obtain representations of fundamental groups into the symmetric group on $n+1$ letters as well as induced representations into other symmetric groups, and are able to obtain topological information from the combinatorics of a triangulation. For instance, if the even triangulation has at most $n$ vertices, then the representation (and hence the fundamental group) is non-trivial. The material described so far can be found in Section~\ref{sec:sym reps and normal}.

In Section~\ref{sec:normal hyper}, we give a new treatment of normal hypersurfaces in $n$--dimensional pseudo-manifolds, including both an algebraic as well as a geometric viewpoint. We show that evenness is a necessary and sufficient condition for certain canonical normal hypersurfaces to be immersed without branching. Moreover, linking the existence of certain embedded hypersurfaces to symmetric representations, we are able to show that if an even triangulation has exactly one vertex, then the fundamental group is non-cyclic.

We then focus on the case of $3$--manifolds. Even triangulations with at least four vertices can be constructed for all closed 3--manifolds, so the cases of interest are those with fewer vertices. In Section~\ref{sec:Properties of 3-manifolds with even triangulations}, we show that if an even triangulation of the closed, orientable 3--manifold $M$ has exactly 
\begin{savenotes}
$$
\left.\begin{array}{l }
\text{three vertices,} \\
\text{two vertices,}\\ 
\text{one vertex,}
\end{array}\right\}
\text{ then $\pi_1(M)$ has an epimorphism onto } 
\left\{\begin{array}{l }
C_2 \text{ or } C_3,\\
C_2 \text{ or } \Alt(4),\footnote{Unless three faces in the triangulation form a spine for $L(3,1)$ and, in particular, $M$ has $L(3,1)$ as a summand in its prime decomposition.} \\ 
C_2\times C_2,\; C_4,\; \Alt(4) \text{ or } \Sym(4).
\end{array}\right.
$$
\end{savenotes}

Having established these necessary conditions for a 3--manifold to admit an even triangulation with few vertices, we also give sufficient conditions. We establish that if $H_1(M,\Z_2) \ne 0$ (i.e.\thinspace $\pi_1(M)$ has an epimorphism onto $C_2$), then $M$ has an even triangulation with two vertices; and if $\pi_1(M)$ has an epimorphism onto $C_2\times C_2$ or $C_4,$ then $M$ has an even triangulation with one vertex. 

There has been much recent interest in $\Z_2$--homology, giving useful topological properties of $3$--manifolds. For instance, $\Z_2$--homology has been studied in the context of searching for incompressible surfaces, bounds on $\Z$--homology and hyperbolic volume---see for example \cite{ACS, CS, Lac1, Lac2, ShWa}. A key result due to Lubotzky~\cite{Lu} is that any complete hyperbolic $3$--manifold of finite volume has a finite sheeted covering so that the rank of $\Z_2$--homology is arbitrarily large. Our results therefore show that any closed hyperbolic 3--manifold has a finite sheeted covering with a 1--vertex even triangulation.

We return to arbitrary dimensions in Section~\ref{sec:Very short hierarchies}, where an interesting connection between very short hierarchies and even triangulations is observed.

In a follow-up paper \cite{BRT}, the authors, together with M. B\"okstedt, construct CAT(0) structures on triangulated $n$--manifolds using Euclidean metrics on $n$-simplices. To verify that the metrics indeed satisfy sufficient conditions to be CAT(0), it turns out to be very convenient to assume that the underlying triangulations are even. 

In a further paper \cite{RT}, the authors generalise the trisections of $4$--manifolds due to Gay and Kirby~\cite{GK} to multisections of $n$--manifolds. A natural duality between triangulations and multisections is constructed---the key observation is again a suitable evenness condition on the triangulations. This also yields a construction of non-positively curved cubed manifolds.

\textbf{Acknowledgements:} The authors are partially supported under the Australian Research Council's Discovery funding scheme (project number DP130103694). The authors thank the Mathematical Sciences Research Centre at Tsinghua University, the Centre for Advanced Study at Warsaw University of Technology, and the Max Planck Institute for Mathematics at Bonn, where parts of this work have been carried out, for their hospitality.


\section{Symmetric representations and normal hypersurfaces}
\label{sec:sym reps and normal}

Motivating examples are discussed in \S\ref{subsec:idea and examples} and revisited in \S\ref{subsec:examples revisited}. The definition of an even triangulation is given in \S\ref{subsec:pseudo-manifolds}. In \S\ref{subsec:Symmetric representations}, we generalise Joswig's group of projectivities from combinatorial manifolds to triangulated pseudo-manifolds, and obtain first topological information from even triangulations with few vertices. In \S\ref{subset:sum reps from partitions}, we introduce symmetric representations arising from partitions. 


\subsection{Examples}
\label{subsec:idea and examples}

The key idea to construct representations is as follows. Given a (possibly singular) triangulation of the $n$--dimensional manifold $M,$ pick one $n$--simplex as a base, label its corners and then \emph{reflect} this labelling across its codimension--one faces to the adjacent $n$--simplices. (This is illustrated for 3--simplices in Figure~\ref{fig:perspectivity}.)
\begin{figure}[h!]
\psfrag{0}{{\small $0$}}
\psfrag{1}{{\small $1$}}
\psfrag{2}{{\small $2$}}
\psfrag{3}{{\small $3$}}
\psfrag{t}{{\small $\varphi_\tau$}}
\psfrag{a}{{\small $\tau$}}
\psfrag{b}{{\small $\varphi_\tau(\tau)$}}
\centering
\includegraphics[scale=0.5]{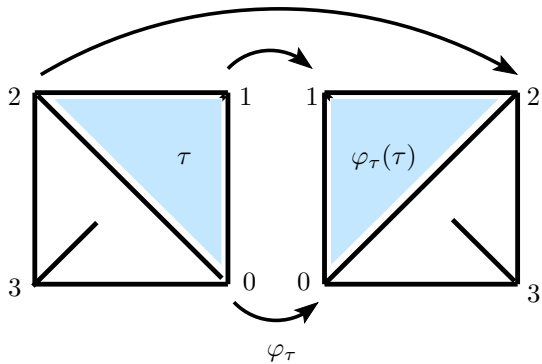}
\caption{Reflecting the labelling across a facet: The labelling of the tetrahedron on the left is reflected across facet $\tau$ to the tetrahedron on the right.} \label{fig:perspectivity}
\end{figure}
This induced labelling is then propagated further across faces, and if one returns to the base simplex, one obtains a permutation of its labels. If all $(n-2)$--simplices have even degree, this yields a representation of the fundamental group of $M$ into $\Sym(n+1).$ This will now be illustrated with some 3--dimensional examples, before we give the formal treatment and produce some general results.

\begin{figure}[h!]
\psfrag{0}{{\small $0$}}
\psfrag{1}{{\small $1$}}
\psfrag{2}{{\small $2$}}
\psfrag{3}{{\small $3$}}
\psfrag{4}{{\small $0'$}}
\psfrag{5}{{\small $1'$}}
\psfrag{6}{{\small $2'$}}
\psfrag{7}{{\small $3'$}}
  \begin{center}
 \subfigure[Quaternionic space]{\includegraphics[scale=0.45]{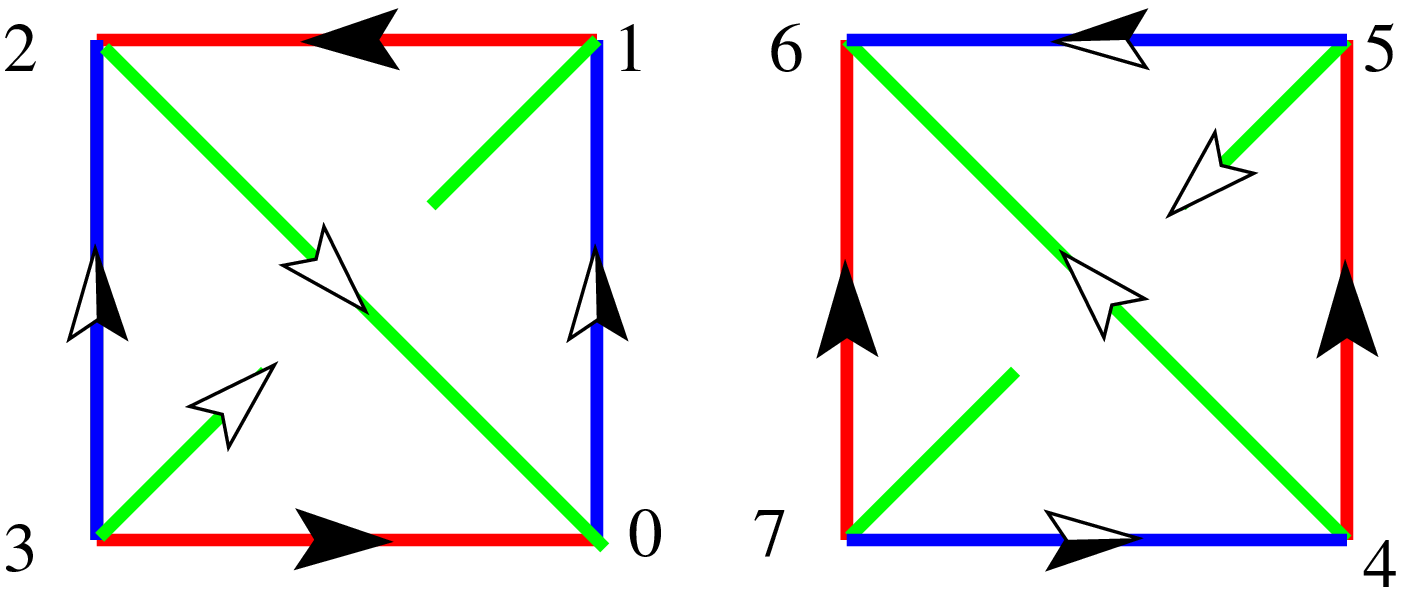}\label{fig:quaternionic}}
    \qquad\qquad
     \subfigure[Figure eight knot complement]{\includegraphics[scale=0.45]{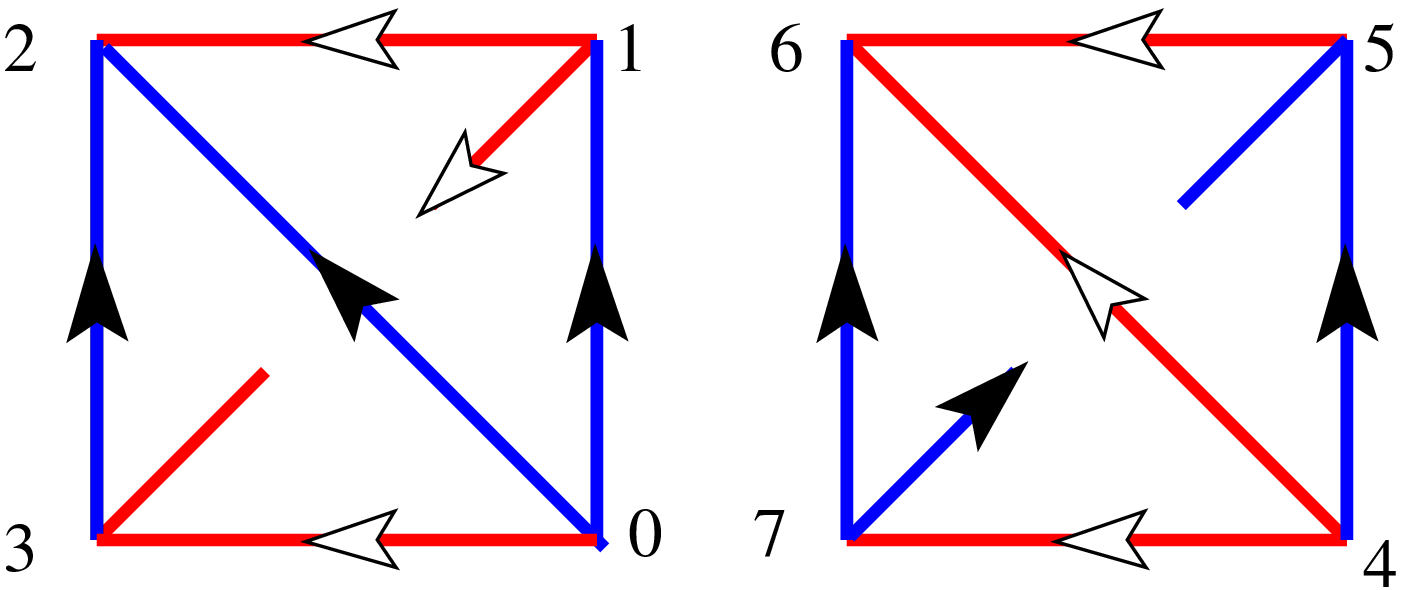} \label{fig:figure8}}
\end{center}
\caption{A closed spherical manifold and a hyperbolic knot complement}
\end{figure}

\begin{example}[(Quaternionic space)]
Quaternionic space $S^3/Q_8$ has a triangulation with two tetrahedra and one vertex; see Figure~\ref{fig:quaternionic}. Denoting the tetrahedra $\sigma_0=[0, 1, 2, 3]$ and $\sigma_1=[0', 1', 2', 3'],$ the faces are identified in pairs:
\begin{align*}
F_3\co [0, 1, 2] \to [3', 0', 1'], \quad 
F_2\co [0, 1, 3] \to [1', 2', 0'], \\
F_1\co [0, 2, 3] \to [2', 0', 3'], \quad
F_0\co [1, 2, 3] \to [3', 2', 1'].
\end{align*}
The labelling of $\sigma_0$ can be {reflected} to $\sigma_1$ via $F_3,$ giving the assigment $3'\to0,$ $0'\to1,$ $1'\to 2$ and $2'\to 3.$ Composing with $F_2^{-1},$ $F_1^{-1},$ and $F_0^{-1}$ respectively results in three permutations of the vertices of $\sigma_0.$ We obtain 
$$F_2^{-1}F_3 \mapsto (02)(13),\quad
F_1^{-1}F_3 \mapsto (03)(12),\quad
F_0^{-1}F_3 \mapsto (01)(23).
$$ 
These compositions of face pairings generate the fundamental group of $S^3/Q_8.$ There are three edges in the triangulation, all having degree four. The product of the reflections across the faces incident with an edge therefore gives the trivial permutation. So
we obtain a natural homomorphism $\pi_1(S^3/Q_8) = Q_8 \to \Sym(4)$ with image isomorphic with $C_2\times C_2.$
\end{example}

\begin{example}[(A lens space)]\label{exa:L(4,1)}
The lens space $L(4,1)$ has a triangulation with just one tetrahedron $[0, 1, 2, 3]$ and face pairings $[0,1,2] \to [3,0,1]$ and $[0,2,3]\to [3,1,2]$ (see Figure~\ref{fig:L(4,1)}). Both give the permutation $(0321).$ Since the degrees of the two edges are $2$ and $4$ respectively, this again corresponds to a homomorphism $\pi_1(L(4,1)) \to \Sym(4),$ which in this case is a homomorphism onto the subgroup isomorphic to $C_4$ generated by $(0321).$ 
\end{example}

The construction also applies to \emph{ideal triangulations}, where the simplices glue up to a \emph{pseudo-manifold} and the complement of the set of vertices is a non-compact manifold whose fundamental group is again generated by certain products of face pairings.

\begin{example}[(The figure eight knot complement)]\label{exa:fig 8}
The complement $M$ of the figure eight knot in $S^3$ has an ideal triangulation with two ideal tetrahedra and one ideal vertex; see Figure~\ref{fig:figure8}. 
Denoting the tetrahedra $\sigma_0=[0, 1, 2, 3]$ and $\sigma_1=[0', 1', 2', 3'],$ the faces are identified in pairs:
\begin{align*}
F_3\co [0, 1, 2] \to [3', 1', 2'], \quad 
F_2\co [0, 1, 3] \to [0', 1', 2'], \\
F_1\co [0, 2, 3] \to [0', 1', 3'], \quad
F_0\co [1, 2, 3] \to [0', 2', 3'].
\end{align*}

The labelling of $\sigma_0$ is again {reflected} to $\sigma_1$ via $F_3,$ giving the assigment $3'\to0,$ $1'\to1,$ $2'\to 2$ and $2'\to 3.$ Composing with $F_2^{-1},$ $F_1^{-1},$ and $F_0^{-1}$ respectively results in three permutations of the vertices of $\sigma_0.$ We obtain 
$$F_2^{-1}F_3 \mapsto (032),\quad
F_1^{-1}F_3 \mapsto (03)(12),\quad
F_0^{-1}F_3 \mapsto (013).
$$ 
In this case, there are two edges of degree 6, so again the products of the reflections across all faces abutting an edge (counted with multiplicity) are trivial, and we obtain a homomorphism from $\pi_1(M) \to \Sym(4)$ with image isomorphic with $\Alt(4).$
\end{example}

\begin{figure}[h!]
\psfrag{0}{{\small $0_0$}}
\psfrag{1}{{\small $1_0$}}
\psfrag{2}{{\small $2_0$}}
\psfrag{3}{{\small $3_0$}}
\psfrag{4}{{\small $0_1$}}
\psfrag{5}{{\small $1_1$}}
\psfrag{6}{{\small $2_1$}}
\psfrag{7}{{\small $3_1$}}
\psfrag{8}{{\small $0_2$}}
\psfrag{9}{{\small $1_2$}}
\psfrag{10}{{\small $2_2$}}
\psfrag{11}{{\small $3_2$}}
\psfrag{12}{{\small $0_3$}}
\psfrag{13}{{\small $1_3$}}
\psfrag{14}{{\small $2_3$}}
\psfrag{15}{{\small $3_3$}}
  \begin{center}
\includegraphics[scale=0.35]{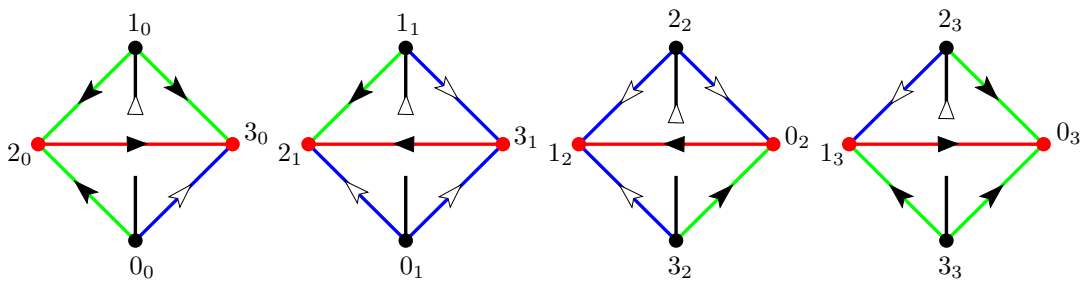}
\end{center}
\caption{The Whitehead link complement}\label{fig:WHL}
\end{figure}

\begin{example}[(The Whitehead link complement)]\label{exa:WHL}
The complement $M$ of the Whitehead link in $S^3$ has an ideal triangulation with four ideal tetrahedra and two ideal vertices; see Figure~\ref{fig:WHL}. Choosing the left-most ideal tetrahedron as a base, the canonical symmetric representation has image a non-normal Klein four group in $\Sym(4)$; namely $\langle (01), (23) \rangle.$
\end{example}


\subsection{Pseudo-manifolds and even triangulations}
\label{subsec:pseudo-manifolds}

Let $\widetilde{\Delta}$ be a finite union of pairwise disjoint, oriented Euclidean $n$--simplices with the standard simplicial structure. Every $k$--simplex $\tau$ in $\widetilde{\Delta}$ is contained in a unique $n$--simplex $\sigma_\tau.$ An $(n-1)$--simplex in $\widetilde{\Delta}$ is termed a \emph{facet} and a $0$--simplex a \emph{corner}. A facet (or corner) $\tau$ has a unique \emph{opposite corner (or facet)} $\sigma_\tau \setminus \tau.$ To simplify notation, we will not distinguish between a singleton and its element.

Let $\Phi$ be a family of orientation-reversing affine isomorphisms pairing the facets in $\widetilde{\Delta},$ with the properties that $\varphi \in \Phi$ if and only if $\varphi^{-1}\in \Phi,$ and every facet is the domain of a unique element of $\Phi.$ The elements of $\Phi$ are termed \emph{face pairings}.

The quotient space $\widehat{M} = \widetilde{\Delta}/\Phi$ with the quotient topology is then a closed, orientable $n$--dimensional pseudo-manifold, and the quotient map is denoted $p\co \widetilde{\Delta} \to \widehat{M}.$ The triple $\tri = ( \widetilde{\Delta}, \Phi, p)$ is a \emph{(singular) triangulation} of $\widehat{M}.$ The set of non-manifold points of $\widehat{M}$ is contained in the $(n-3)$--skeleton. (See Seifert-Threfall~\cite{SeiThr}.) If $n=2,$ then $\widehat{M}$ is a surface.  We will always assume that $\widehat{M}$ is connected. In the case where $\widehat{M}$ is not connected, the results of this paper apply to its connected components.

The quotient space $\widehat{M}$ is studied via the map $p\co \widetilde{\Delta} \to \widehat{M}.$ The image of a $k$--simplex under $p$ is termed a \emph{$k$--singlex} in $M.$ The \emph{degree} of the $k$--singlex $\tau$ is the degree of the restriction $p \co p^{-1}(\tau)\to \tau.$ We term $\tau$ \emph{even} if its degree is even, and we term the triangulation of $\widehat{M}$ \emph{even} if every $(n-2)$--singlex is even.

Denote $\widehat{M}^{(k)}$ the image of the $k$--skeleton of $\widetilde{\Delta}$ under the projection map, and $M = \widehat{M}\setminus \widehat{M}^{(0)}.$ The pseudo-manifold $\widehat{M}$ is often referred to as the \emph{end-compactification} of $M.$ Denoting the restrictions of $\Phi$ and $p$ to $ \widetilde{\Delta} \setminus \widetilde{\Delta}^{(0)}$ by the same letters, the triple $\tri = ( \widetilde{\Delta}\setminus \widetilde{\Delta}^{(0)}, \Phi, p)$ is an \emph{ideal (singular) triangulation} of $M.$
The dual $(n-1)$--skeleton in $\widehat{M}$ is called a \emph{spine} for $M,$ as $M$ retracts onto this spine. Hence these two have isomorphic fundamental groups, whilst the fundamental group of $\widehat{M}$ is a quotient thereof. The \emph{dual graph} or \emph{dual 1--skeleton} of the triangulation is the 1--skeleton of the dual $(n-1)$--skeleton. It has one vertex for each $n$--singlex (which is identified with its barycentre) and one edge for each $(n-1)$--singlex.

As indicated, the adjective \emph{singular} is usually omitted, and we will not need to distinguish between the cases of a simplicial or a singular triangulation.


\subsection{Symmetric representations}
\label{subsec:Symmetric representations}

Given the facet $\tau,$ denote $\varphi_\tau$ the face pairing with domain $\tau.$
Following Joswig~\cite{Joswig2002}, the facet $\tau$ defines the \emph{perspectivity} $p_\tau \co \sigma_\tau \to \sigma_{\varphi_\tau(\tau)}$ by:
\begin{equation}
 \tau^{n-2} \mapsto \begin{cases} \varphi_\tau(\tau^{n-2}) & \text{if } \tau^{n-2} \subseteq \tau, \\
 						\sigma_{\varphi_\tau(\tau)} \setminus {\varphi_\tau(\tau)} & \text{otherwise.}
			\end{cases}			
\end{equation}
A perspectivity acts like a reflection across a facet (see Figure~\ref{fig:perspectivity}). The definition implies that $$p_{\varphi_\tau(\tau)} \circ p_\tau\co \sigma_\tau \to \sigma_\tau$$ is the identity.

A \emph{facet path} from facet $\tau$ to $\tau'$ is a finite sequence
\begin{equation}
\gamma = ( \tau_0, \tau'_1, \tau_1, \ldots ,  \tau'_k, \tau_k, \tau'_{k+1}),
\end{equation}
where $\tau_0 = \tau,$ $\tau'_{k+1} = \tau',$ and $\tau'_j, \tau_j$ are distinct facets contained in a common $n$--simplex $\sigma_j$ and $\tau'_{j+1} = \varphi_{\tau_j}(\tau_j)$ for all $j\in \{0,\ldots k\}.$ Again following Joswig~\cite{Joswig2002}, we define the \emph{projectivity} 
$$p_\gamma \co \sigma_{\tau_0} \to \sigma_{\tau'_{k+1}}$$ by
$$p_\gamma = p_{\tau_{k}} \circ \ldots \circ p_{\tau_{1}}\circ p_{\tau_{0}}.$$
In the case where $\sigma_0=\sigma_{k+1},$ the facet path $\gamma$ is termed a \emph{facet loop}. In this case, the projectivity $p_\gamma$ is a permutation of the vertices of $\sigma_0,$ and termed a \emph{projectivity of $\widehat{M}$ based at $\sigma_0.$} 

Using concatenation as the group operation, the set of all projectivities of $\widehat{M}$ based at $\sigma_0$ forms a subgroup of $\Sym(\sigma_0),$ which is denoted $\Pi(\widehat{M}, \sigma_0).$

If $\gamma$ is any facet path from $\sigma_0$ to another $n$--simplex $\sigma,$ then 
$$\Pi(\widehat{M}, \sigma) = p_{\gamma} \Pi(\widehat{M}, \sigma_0) p_{\gamma^{-1}}.$$
It follows that $\Pi(\widehat{M}, \sigma)$ can be identified with a subgroup of the symmetric group on $n+1$ letters, $\Sym({n+1}),$ up to inner automorphisms of $\Sym({n+1}).$

From now on, \textbf{we will only address the case} $\mathbf{n\ge 3},$ as modifications are required for surfaces, and our methods do not yield interesting results in this case.

Let $x_0$ be the image in $\widehat{M}$ of the barycentre of $\sigma_0.$ Then for each loop in $\widehat{M}$ based at $x_0$ and contained in the dual graph of the triangulation, we obtain a facet loop based at $\sigma_0$ obtained by recording the sequence of facets which the loop crosses. Whence there is a well-defined projectivity $p_\gamma$ based at $\sigma_0$ associated with any loop $\gamma$ in the dual 1--skeleton. The fundamental group of $M$ is finitely presented with generators represented by loops in the dual 1--skeleton, and relators arising from the cells in the dual 2--skeleton. 

%

\begin{lemma}[Canonical symmetric representation]\label{lem:canonical rep}
Suppose $n\ge3.$ The assignment 
$$\pi_1(M, x_0)\ni[\gamma] \to p_\gamma \in \Sym(\sigma_0)$$ 
is a well-defined homomorphism if and only if the triangulation of $\widehat{M}$ is even.
\end{lemma}


\begin{proof}
Since every element in the fundamental group is represented by a simplicial path in the dual 1--skeleton, and all relations arise from the cells in the dual 2--skeleton, it remains to show that the homomorphism is well defined if and only if the degree of each $(n-2)$--cell in $\widehat{M}$ is even.

A loop $\gamma$ in the dual 1--skeleton based at $x_0$ and abutting the cell $c$ in the dual 2--skeleton represents the trivial element in $\pi_1(M; x_0).$ The associated element $\rho(\gamma) \in \Sym(\sigma_0)$ is trivial if and only if the degree of the $(n-2)$--cell dual to $c$ is even. Similarly, given any cell $c$ in the dual 2--skeleton, choose a path $\alpha$ in the dual 1--skeleton from $x_0$ to the barycentre of an $n$--singlex incident with $c$ and a loop $\beta$ abutting $c.$ Then $0 = [\alpha\beta][\alpha]^{-1},$ and the associated permutation is trivial if and only if the degree of the $(n-2)$--cell dual to $c$ is even.
Hence all relators in the fundamental group give trivial images in $\Sym(\sigma_0)$ if and only if the triangulation is even. 
\end{proof}

\begin{remark}
The homomorphism $\pi_1(M, x_0)\ni[\gamma] \to p_\gamma \in \Sym(\sigma_0)$ factors through $\pi_1(\widehat{M}, x_0)$ if the link of each vertex has trivial fundamental group (for instance, is an $(n-1)$--sphere). Simple examples where it does not factor are given by the ideal triangulations of knot and link complements in the 3--sphere (see Example~\ref{exa:fig 8}).
\end{remark}

\begin{remark}
If a triangulation is not even, one can use the null-homotopic loops with non-trivial projectivities to construct canonical branched coverings; see Izmestiev and Joswig \cite{IJ-2003} for an application.
\end{remark}

\begin{proposition}[few vertices implies non-trivial representation]\label{pro:few vert give non-trivial}
If the vertices $v$ and $v'$ of the $n$--simplex $\sigma_0$ have the same image in $\widehat{M}$ under the map $p\co \widetilde{\Delta} \to \widehat{M},$ then there is $[\gamma] \in \pi_1(M, x_0)$ with $p_\gamma(v)=v'.$
In particular, if an even triangulation has fewer than $n+1$ vertices, then the canonical symmetric representation has non-trivial image, and so the fundamental group of $M$ is non-trivial.
\end{proposition}

\begin{proof}
The identification of $v$ and $v'$ arises from some facet loop $\gamma$ based at $\sigma_0.$ If $v\neq v',$ then the associated permutation $p_\gamma$ is non-trivial. 
\end{proof}

For example, this shows that the $n$--sphere does not admit an even triangulation with fewer than $n+1$ vertices. However the double of an $n$--simplex is an even degree triangulation of $S^n$ with exactly $n+1$ vertices. 

\begin{corollary}
If the canonical symmetric representation has trivial image, then the vertices in $\widehat{M}$ can be consistently labelled $\{1,2,\ldots, n+1\}$ and there must be at least $n+1$ vertices.
\end{corollary}


\subsection{Examples (Revisited)}
\label{subsec:examples revisited}

The next step in our program is again introduced with an informal discussion and some 3--dimensional examples, which are then distilled into an algebraic theory in arbitrary dimensions.

Given a closed, triangulated 3--manifold $M,$ place three quadrilateral discs in each tetrahedron, one of each type as shown in Figure~\ref{fig:quads}, such that the result is a (possibly branched immersed) surface in $M.$ 
\begin{figure}[h!]
\psfrag{0}{{\small $0$}}
\psfrag{1}{{\small $1$}}
\psfrag{2}{{\small $2$}}
\psfrag{3}{{\small $3$}}
\begin{center}
\includegraphics[scale=0.4]{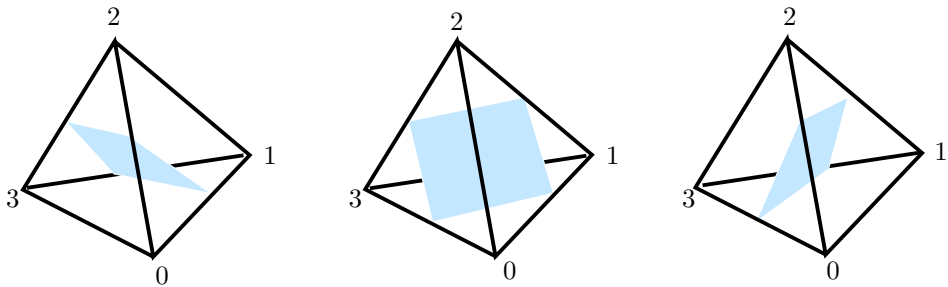}
\caption{The three types of quadrilaterals in a 3--simplex}\label{fig:quads}
\end{center}
\end{figure}

If one chooses a labelling of the three quadrilateral discs, one can again propagate the labelling chosen in some base tetrahedron across all faces to a labelling of the quadrilaterals in the adjacent tetrahedra. As before, if the degree of each edge in the triangulation is even, there is an associated representation into the symmetric group $\Sym(3)$ on three letters. (Even degree also implies that the surface has no branch points, and hence is immersed.) We will show that this representation, denoted $\prho$, is \emph{induced} by the canonical symmetric representation $\rho$, and that its image is related to the number of components of the surface. But first, we discuss the geometric manifestation of this induced representation in our examples.

\begin{example}[(Quaternionic space)]
The quadrilateral surface in the triangulation of quaternionic space has three components; each is an embedded Klein bottle meeting each tetrahedron in exactly one quadrilateral disc. If one chooses a labelling in a base tetrahedron and propagates it, one obtains the trivial permutation associated to each generator of the fundamental group since otherwise one would obtain a contradiction to the fact that each component meets each tetrahedron in exactly one quadrilateral disc.
\end{example}

\begin{example}[(A lens space)]
The quadrilateral surface in $L(4,1)$ has two components: one is an embedded Klein bottle and the other is an immersed projective plane; see Figure~\ref{fig:L(4,1)}. Labelling the quadrilateral discs in $\sigma_0$ respectively $r$ (red), $b$ (blue) and $g$ (green), one obtains labels of the normal arcs on each face and is interested on the effect of the face pairings on the labels. Suppose that the labels are $[0, b; 1, r; 2, g]$ on face $[0,1,2].$ We have $[0, b; 1, r; 2, g] \to [3, g; 0, r; 1, b]$ and $[0, g; 2, b; 3, r]\to [3,b; 1,g; 2, r].$ Whence $r$ is fixed (corresponding to the Klein bottle) and $g$ and $b$ are transposed by each face pairing, giving a representation $\pi_1(L(4,1)) \to \Sym(3)$ with image isomorphic with $C_2.$
\end{example}
\begin{figure}[h!]
\psfrag{0}{{\small $0$}}
\psfrag{1}{{\small $1$}}
\psfrag{2}{{\small $2$}}
\psfrag{3}{{\small $3$}}
\centering
\includegraphics[scale=0.5]{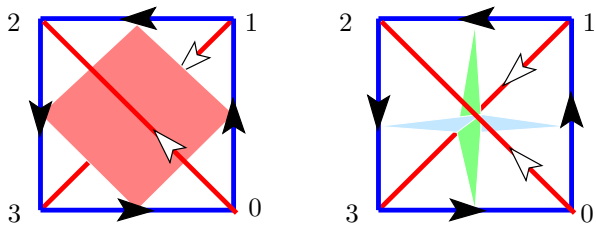}
\caption{The embedded Klein bottle in $L(4,1)$ is shown on the left, and the immersed projective plane on the right.} \label{fig:L(4,1)}
\end{figure}

\begin{example}[(The figure eight knot complement)]
The quadrilateral surface has one component of Euler characteristic $-2.$ Choosing a labelling in one tetrahedron and propagating, it is not difficult to compute that the image is isomorphic with $C_3.$ The whole class of once-punctured torus bundles, into which this example fits, is discussed in the next example.
\end{example}

\begin{example}[(Once-punctured torus bundles)]\label{exa:torus bundles}
A nice class of examples is the canonical triangulations of the once-punctured torus bundles over the circle. (See \cite{fg}.)
These all have the property that they are ideal triangulations for which all edges are of even degree and there is a single ideal vertex. There are different behaviours of the symmetric representation $\rho$. For the figure-eight knot complement given from above, the monodromy of the bundle has trace $3$ and the canonical triangulation gives rise to the representation $\prho$ into $\Sym(3)$ with image $C_3$ and the symmetric representation $\rho$ with image $\Alt(4)$. In this case, there are no elements of $H_1(M,\Z_2)$ other than coming from the meridian which is a boundary element. There are many other examples of this type in the class of once-punctured torus bundles. 

The monodromy of a punctured torus bundle $M$ maps to an element of $SL(2,2) \cong \Sym(3)$. It is not difficult to verify for the canonical triangulations that the induced symmetric representation $\prho$ has image which is cyclic of order the same as that of the image of the monodromy in $SL(2,2)$. The reason is that if we pass to the cyclic covering space $\widetilde{M}$ corresponding to the kernel, then clearly the monodromy has image the identity element in  $SL(2,2)$. But then $H_1(\widetilde{M},\Z_2)$ has rank $3$ and the quadrilateral surface splits into three embedded non-orientable components. In general, the quadrilateral surface has one, two or three components which are all non-orientable, and the  symmetric representation $\rho$ has image $\Alt(4)$, $D_4$ or $C_2\times C_2$ respectively. 
\end{example}

\begin{remark}[(Minimal triangulations)]
Even triangulations and quadrilateral surfaces seem to play a special role in the search for minimal triangulations of closed 3--manifolds---see \cite{JRT,JRT-1,JRT-2}. However, it is \emph{not} true that if a $3$--manifold satisfies the conditions required to have a one-vertex even triangulation, then amongst its minimal triangulations there is one of this type; see Example~\ref{exa:minimal counter}.
\end{remark}


\subsection{Symmetric representations from partitions}
\label{subset:sum reps from partitions}

Returning to the case of arbitrary dimension and continuing in the notation from \S\S\ref{subsec:pseudo-manifolds}--\ref{subsec:Symmetric representations}, suppose that the triangulation of $\widehat{M}$ is even. The $n$--simplex $\sigma_0$ is identified with the set $\{1, 2, ..., n+1\},$ and we denote the canonical symmetric representation $\rho\co \pi_1(M)\to \Sym(n+1).$

Taking a partition of $\sigma_0,$ one can act on this by the induced representation. For the purpose of this paper, the treatment is restricted to partitions into two sets, one of size $k$ and the other of size $n-k+1.$ For instance, 
$$\{\;\{1, 2,..., k\},\;\{ k+1,..., n+1\}\;\} \mapsto \{\;\{\rho(1), \rho(2),..., \rho(k)\},\;\{ \rho(k+1),..., \rho(n+1)\}\;\}.$$
This gives an induced representation $\prho={\rho}_{k\mid n-k+1}\co \pi_1(M)\to\Sym(N)$, where $N = {n+1 \choose k}$ unless the two sets are of equal size, in which case $N = \frac{1}{2}{n+1 \choose k}$ since the order of the two sets in the partition does not matter. The order of the two sets in the partition being irrelevant also implies that one may assume $k \le \frac{n+1}{2}.$ Moreover, we require $k>1,$ since the induced representation for $k=1$ is conjugate to $\rho.$

It follows that there are $\lfloor  \frac{n-1}{2} \rfloor$ induced representations from partitions into two sets. For $n\ge 5,$ there are at least two. For $n\ge 4,$ we always have $n+1 < N.$ We now discuss the cases of lowest dimension.


For $n=4$, there is a unique induced representation, ${\rho}_{2\mid 3},$ arising from a subdivision into sets of sizes 2 and 3 respectively, with target $\Sym(10).$

For $n=3$, there is a unique induced representation, arising from a subdivision into two sets of size 2. The target of the induced representation is $\Sym(3).$ This is the only case where $N < n+1.$ Moreover, it is easy to see the relationship between the canonical and the induced representations:

\begin{proposition}\label{pro:sym factors for n=3}
Let $K$ be the Klein $4$--subgroup of $\Sym(4)$ and $p\co \Sym(4) \to \Sym(3)$ be the epimorphism with kernel $K$. Then ${\rho}_{2\mid 2}=p \circ \rho$.
\end{proposition}

\begin{proof}
Identifying the partition $\{ \{a, b\}, \{c, d\}\}$ with the permutation $(ab)(cd),$ the action of $\Sym(4)$ on the partitions is identified with its action on $K$ by conjugation. The kernel of this action is $K.$
\end{proof}



\section{Normal hypersurface theory}
\label{sec:normal hyper}

We define normal hypersurface theory in arbitrary dimensions, and show that eveness is a necessary and sufficient condition for immersion of certain hypersurfaces, and the symmetric representations give such conditions for embedding. This also gives further results linking the topology of the manifold with the combinatorics of an even triangulation.


\subsection{Algebraic notions}

Continuing in the notation from \S\ref{subsec:pseudo-manifolds}, but not assuming that our triangulation is even, we now define a generalisation of normal surface theory to arbitrary dimensions.

Let $\sigma = \{ 1, 2, \ldots, n+1\}$ be an $n$--simplex. A \emph{$(k, n-k+1)$--normal disc type} or \emph{normal disc type} in $\sigma$ is a partition into two subsets
$$\{\;\{a_1, a_2,..., a_k\},\;\{ a_{k+1},..., a_{n+1}\}\;\},$$
where we assume that $1\le k \le \frac{n+1}{2}.$ Geometrically, a $(k, n-k+1)$--normal disc is viewed as a properly embedded linear cell in the standard $n$--simplex separating the vertices into two sets as indicated by the partition (see below). For instance, $(1,2)$--normal discs are usually termed \emph{normal arcs}, $(1,3)$--normal discs are \emph{normal triangles} and $(2,2)$--normal discs are \emph{normal quadrilaterals}. The normal discs defined here correspond to normal isotopy classes of the usual geometric objects.

Introduce one real variable $x_\nu$ for each normal disc type $\nu$ in $\widetilde{\Delta}.$ Notice that a normal disc in $\sigma$ meets any subsimplex of $\sigma$ in a normal disc or the empty set. Associated with each normal disc type $\mu$ of each facet $\tau$ of $\widetilde{\Delta},$ there is one \emph{matching equation $E_\mu$} defined as follows:
\begin{equation}\tag{$E_\mu$}
\sum_{\nu\;:\;\nu \cap \tau = \mu} x_\nu \quad = \quad \sum_{\nu\;:\;\nu \cap \varphi_\tau(\tau) = \varphi_\tau(\mu)} x_\nu,
\end{equation}
where in both sums $\nu$ ranges over all normal disc types in $\widetilde{\Delta}.$ The matching equations for $\mu$ and $\varphi_\tau(\mu)$ are identical. 

Given \emph{any} $(k, n-k+1)$--normal disc type $\nu,$ there is a solution $\mathbf{x}$ to the matching equations with entries in $\{0,1\}.$ This solution can be generated using perspectivities as follows. For any facet $\tau$ of $\sigma,$ the perspectivity $p_\tau$ determines the partition $p_\tau(\nu)$ of $\sigma_{\varphi_\tau(\tau)},$ and these two partitions satisfy the matching equation $E_\mu$ for each normal disc type $\mu$ of $\tau.$ One now needs to propagate this process, inductively applying perspectivities until all matching equations are satisfied (in particular, the perspectivity $p_\tau$ may need to be applied multiple times during this process). All normal disc types generated by this process are $(k, n-k+1)$--normal disc types, and more than one normal disc type supported by $\sigma$ may be obtained. However, the process will terminate as there are only finitely many $(k, n-k+1)$--normal disc types, and the resulting collection of normal disc types satisfies all matching equations. This process clearly generates the solution to the matching equations having smallest weight amongst all solutions with $x_\nu = 1$ and only involving $(k, n-k+1)$--normal disc types. In particular, this argument yields the following lemma:

\begin{lemma}[$(k,n-k+1)$--solutions]\label{lem:(k,n-k+1) solutions}
Define $\mathbf{x} = (x_\nu)$ by $x_\nu = 1$ if $\nu$ is of type $(k,n-k+1),$ and $x_\nu=0$ otherwise. Then $\mathbf{x}$ is a solution to the matching equations.
\end{lemma}

The lemma gives a lower (albeit poor) bound on the dimension of the solution space of the matching equations. This space is well understood for dimensions 3 (see \cite{KR, T}) and 4 (see \cite{BFR}), and a general treatment would be a rewarding task. The above set-up is inspired by \cite{TONS}. 


\subsection{Geometric notions}

Whilst the partition into two subsets is convenient to work with algebraically, we now define \emph{normal discs}, which geometrically realise the partitions. Suppose that each simplex $\sigma$ in $\widetilde{\Delta}$ is isometric with a regular Euclidean simplex, and identify its \emph{set of vertices} with $\{ 1, 2, \ldots, n+1\}.$ The normal disc type $\nu = \{\;\{a_1, a_2,..., a_k\},\;\{ a_{k+1},..., a_{n+1}\}\;\}$ is represented by the convex hull $\overline{\nu}$ of the set of midpoints of all edges in $\sigma$ that have one endpoint in $\{a_1, a_2,..., a_k\}$ and the other in $\{ a_{k+1},..., a_{n+1}\}.$ We claim that this is topologically an $(n-1)$--dimensional ball, namely the intersection of $\sigma$ with a hyperplane disjoint from the two subsimplices given by the partition and meeting the complementary edges in their mid-points. The following argument is due to Jonathan Spreer. We may assume that $\sigma$ lies in the hyperplane $x_1 + \ldots + x_{n+1}=1$ in $\R^{n+1}$ and has its vertices on the coordinate axes (with $x_j$ corresponding to $a_j$). The above set of midpoints then lies in the intersection of the two hyperplanes defined respectively by:
$$
x_1 + \ldots + x_{k}=\frac{1}{2}
\quad\text{and}\quad
x_{k+1} + \ldots + x_{n+1}=\frac{1}{2}.
$$
We call $\overline{\nu}$ the \emph{standard normal $(n-1)$--disc} in $\sigma$ representing $\nu.$ More generally, a normal $(n-1)$--disc of type $\nu$ is a properly embedded PL $(n-1)$--cell that is normally isotopic to $\overline{\nu}.$
Here, a \emph{normal isotopy (homotopy)} is an isotopy (homotopy) of $\widehat{M}$ which leaves all singlices invariant.

\begin{lemma}[Branched immersed normal hypersurfaces]
Each solution $\mathbf{x} = (x_\nu)_\nu$ with non-negative integers to the system of all matching equations corresponds to a branched immersion of an $(n-1)$--dimensional pseudo-manifold into $\widehat{M}.$ We will call this a branched immersed \emph{normal hypersurface} of $\widehat{M}.$
\end{lemma}

\begin{proof}
Given the solution $\mathbf{x},$ we will build a (non-unique) $(n-1)$--dimensional pseudo-manifold $H$ and a branched immersion $f\co H \to \widehat{M}.$ The argument is a generalisation of the one given in \cite{CT} for the case $n=3.$

The solution $\mathbf{x}$ determines an abstract collection $\mathcal{C}$ of normal discs by taking $x_\nu$ copies of the normal disc corresponding to the standard normal disc type $\nu.$ If $\nu$ meets the facet $\tau$ in $\mu,$ then the standard normal disc $\overline{\nu}$ meets $\tau$ in the standard normal disc $\overline{\mu}.$ The face pairing $\varphi_\tau$ gives an affine isomorphism between $\overline{\mu}$ and $\overline{\varphi_\tau(\mu)}.$ Since $\mathbf{x}$ satisfies the equation $E_\mu,$ we can pair all boundary facets of normal discs in $\mathcal{C}$ corresponding to $\mu$ bijectively with the boundary faces of normal discs in $\mathcal{C}$ corresponding to $\varphi_\tau(\mu).$ Doing this for all facets gives an $(n-1)$--dimensional CW complex $H$ with empty boundary together with a natural map of $H$ to $\widehat{M}$ mapping each $(n-1)$--cell in $H$ to the image of a normal disc in $\widehat{M}.$ This map is an immersion on the complement of the $(n-2)$--skeleton of $H,$ and transverse to the 1--skeleton. It hence is a branched immersion. Since each normal $(n-1)$--ball can be subdivided into $(n-1)$--simplices (for instance, by coning to its barycentre),  $H$ is an $(n-1)$--dimensional pseudo-manifold.
\end{proof}

We say that partitions $\nu$ and $\omega$ of the vertices of $\sigma$ are \emph{compatible} if we can write $\nu = \{ A, B\}$ and $\omega = \{ C, D\}$ with $C \subseteq A$ and $B \subseteq D.$ Since a normal homotopy fixes all vertices of $\sigma,$ there are disjoint normal discs of types $\nu$ and $\omega$ in $\sigma$ if and only if $\nu$ and $\omega$ are compatible.

\begin{lemma}[Embedded normal hypersurfaces]
Suppose $\mathbf{x} = (x_\nu)_\nu$ is a solution with non-negative integers to the system of all matching equations. Then $\mathbf{x}$ corresponds to an embedding of an $(n-1)$--dimensional pseudo-manifold into $\widehat{M}.$ if and only if for each $n$--simplex $\sigma$ in $\widetilde{\Delta}$ and all partitions $\nu$ and $\omega$ of the vertices of $\sigma,$ $x_\nu\neq 0$ and $x_\omega \neq 0$ only if $\nu$ and $\omega$ are compatible.
\end{lemma}

\begin{proof}
The forwards direction follows from the discussion before the lemma. For the converse, notice that if for each $n$--simplex $\sigma$ we have $x_\nu\neq 0$ and $x_\omega \neq 0$ only if $\nu$ and $\omega$ are compatible, then we can place a pairwise disjoint collection of normal discs in $\sigma,$ taking $x_\nu$ parallel copies of $\overline{\nu}$ for each $\nu.$ The matching equations ensure that (after a normal isotopy of each $n$--simplex) these can be glued along the facets using the face pairings.
\end{proof}


\subsection{Conditions for immersion or embedding}

For the $(k,n-k+1)$--solutions described in Lemma~\ref{lem:(k,n-k+1) solutions}, there is a \emph{unique} branched immersed normal hypersurface with this coordinate, as the weights are all zero or one. We call these the \emph{$(k,n-k+1)$--hypersurfaces} in $\widehat{M}.$ The $(1,n)$--hypersurface is the union of all vertex links, and can hence be represented by an embedding. The next result shows that even triangulations give a necessary and sufficient condition to rule out branching when $k>1.$

\begin{lemma}[$(k,n-k+1)$--hypersurface is immersion]
Suppose $k>1.$ Then the $(k,n-k+1)$--hypersurface with normal coordinate $\mathbf{x}$ is an immersion if and only if the triangulation is even.
\end{lemma}

\begin{proof}
Suppose $\tau^{n-2}$ is an $(n-2)$--singlex in $\widehat{M},$ and let $\sigma$ be an $n$--singlex incident with $\tau^{n-2}.$ Then there is a $(k, n-k+1)$--normal disc type $\nu$ in $\sigma$ with the property that its restriction to $\tau^{n-2}$ is a $(k-1, n-k)$--normal disc type. The $(k, n-k+1)$--hypersurface is not branched along the intersection of $\tau^{n-2}$ with $\nu$ if and only if $\nu$ is identified to itself upon a single circuit around $\tau^{n-2}.$ But since the two vertices of $\nu$ not in $\tau^{n-2}$ are contained in different sets of the partition, and identifications along facets are realised by perspectivities, this is the case if and only if the degree of $\tau^{n-2}$ is even. Whence the branch locus of the hypersurface is empty if and only if the degree of each $(n-2)$--face is even.
\end{proof}

Now that evenness is a necessary and sufficient condition for immersion, we can link the induced representations to embeddings.
The induced representation $\overline{\rho}={\rho}_{k|n-k+1}\co\pi_1(M)\to \Sym(N),$ where $N$ is as in \S\ref{subset:sum reps from partitions}, clearly leaves the $(k, n-k+1)$--hypersurface invariant, and the next result determines its relationship with the hypersurface.

\begin{proposition}[trivial induced representation implies embedding]\label{pro:trivial induced implies embedding}
Suppose $\widehat{M}$ has an even triangulation and $k>1.$ Then ${\rho}_{k|n-k+1}$ leaves each component of the $(k,n-k+1)$--hypersurface $Q$ invariant. Moreover, $\im ({\rho}_{k|n-k+1})$ is trivial if and only if $Q$ has $N$ components. Each of these components is an embedded normal hypersurface in $\widehat{M}.$
\end{proposition}

\begin{proof}
If $Q_0$ is a component of $Q$ and meets some $n$--singlex, $\sigma,$ in $i$ normal discs, $i \in \{ 0,1,\ldots, N\},$ then the same is true for each $n$--singlex meeting $\sigma$ in a face since each $(k,n-k+1)$--normal disc meets each facet of each $n$--singlex. Since $M$ is connected, the same is true for every $n$--singlex in the triangulation (and in particular $i \neq 0$ if $Q_0\neq \emptyset$). It follows immediately that $Q$ has $N$ components if and only if each component $Q_0$ of $Q$ meets each $n$--singlex in one normal disc. This clearly also implies if $Q$ has $N$ components that each such component is embedded. 

It was observed that perspectivities encode the information about identifications of normal discs along their boundary faces. It follows that two $(k,n-k+1)$--normal disc types in $\sigma_0$ are in the same component of $Q$ if and only if there is a projectivity taking one to the other. The action of projectivities on the $(k,n-k+1)$--normal disc types in $\sigma_0$ is precisely the action of ${\rho}_{k|n-k+1}.$
\end{proof}

\begin{corollary}\label{cor:trivial induced rep gives non-cyclic}
Suppose $M$ is a closed $n$--manifold. If $M$ has an even triangulation with one vertex and $\im ({\rho}_{k|n-k+1})$ is trivial for some $k>1,$ then $H_1(M; \Z_2)$ has rank at least $k.$ In particular, $\pi_1(M)$ is non-cyclic.
\end{corollary}

\begin{proof}
The edges in $M$ are loops and generate $\pi_1(M),$ where the vertex of the triangulation is chosen as a base-point. Each component $Q_j$ of the embedded $(k,n-k+1)$--hypersurface defines an element of $H_{2k-2}(M; \Z_2),$ and meets each edge loop in either one or no point. We therefore obtain the element $\alpha_j$ of $H^1(M; \Z_2),$ which is Poincar\'e dual to $Q_j,$ by taking the intersection numbers of $Q_j$ with the edges. Considering the edge loops incident with a single vertex $v$ of a single $n$--simplex $\sigma$ in $\widetilde{\Delta}$ gives the result as follows. There are $n$ edges of $\sigma$ that are incident with $v.$ Choosing a $(k-1)$--simplex in $\sigma$ incident with $v$ amounts to choosing $k-1$ edges incident with $v.$ The dual cohomology class sends these $k-1$ edges to $0$ and the remaining $n-k+1$ edges to $1.$ We may identify the restriction of this homology class with an element of $\Z_2^{n},$ with the coordinate axes represented by the edges incident with $v.$ A lower bound on the dimension of $H^1(M; \Z_2)$ is then given by a lower bound on the rank of the matrix whose rows are all vectors having $k-1$ entries equal to zero and $n-k+1$ entries equal to one. Since $n-k+1\ge k-1,$ it suffices to restrict to the submatrix whose rows are all vectors having $k-1$ entries equal to zero and $k-1$ entries equal to one. This has rank at least $k,$ as can be seen by considering all vectors of the form $(0, \ldots, 0, 1 , \ldots , 1, 0, \ldots, 0),$ where the first or the last sequence of zeros may be empty.
\end{proof}

\begin{remark}
More generally, a function in the number of embedded components of the $(k, n-k+1)$--hypersurface gives a lower bound on the rank of $H_1(M; \Z_2);$ this number is in turn the number of $(k, n-k+1)$--disc types that is fixed by ${\rho}_{k|n-k+1}.$ In particular, if $M$ has finite fundamental group of odd order, then no component of the hypersurface is embedded.
\end{remark}

\begin{corollary}
Suppose $M$ is a closed $n$--manifold with an even triangulation with one vertex. If the fundamental group of $M$ is finite and of order a prime power, then the prime is less than or equal to $n+1.$
\end{corollary}

\begin{proof}
This follows from the fixed point congruence for $p$--groups: if the prime is bigger than $n+1,$ then ${\rho}_{k|n-k+1}$ fixes at least one $(k,n-k+1)$--normal disc type since $p$ does not divide the total number of $(k,n-k+1)$--normal disc types. Whence $H_1(M; \Z_2)$ has rank at least $1,$ and so $p=2<n+1.$ 
\end{proof}

We end this section with some examples, which focus on the particularly interesting case, where the partition is (as close as possible to) half-half.

\begin{example}
Suppose $\widehat{M}$ is a 3--dimensional pseudo-manifold. Then $Q$ is obtained by placing three quadrilateral discs in each tetrahedron, one of each type, such that the result is a (possibly branched immersed) normal surface. The above implies that $Q$ consists of three embedded surfaces if and only if every edge has even degree and the induced representation $\rho_{2|2}$ has trivial image. In our applications to 3--manifolds in \S\ref{sec:Properties of 3-manifolds with even triangulations}, we will pass to a finite covering space corresponding to the kernel of $\rho_{2|2}.$ In this covering space, the hypersurface lifts to $3$ connected components, and each component is hence an embedded hypersurface in the cover.
\end{example}

\begin{example}
Consider the case of a $4$--manifold. In the covering space, in which $\rho_{2|3}$ has trivial image, the components of $Q$ are $3$--manifolds. These 3--manifolds are unions of $3$--cells, which are products $\sigma^2 \times \sigma^1$, i.e.\thinspace triangular prisms, and divide the $4$--manifold into two regions. One is a $4$--dimensional handlebody, i.e.\thinspace a connected sum of copies of $S^1 \times B^3$. The other has a $2$--dimensional spine. There are ten such triangular prisms corresponding to the $2|3$ partitions of $5$ vertices.
\end{example}

\begin{example}\label{exa:even 5}
Consider the case of $5$--dimensional pseudo-manifolds. The hypersurface $Q$ consists of $4$--cells which are products $\sigma_2 \times \sigma_2$. The boundary of such a $4$--cell is a copy of $S^3$ tiled with six copies of a triangular prism $\sigma_2 \times I$. These prisms form two solid tori giving the standard Heegaard splitting of $S^3$. The $4$--cells are glued together along the triangular prisms to form the hypersurface $Q$. In the covering space, where the symmetric representation has trivial image, each component $Q_0$ of $Q$ is an embedded $4$--manifold with an even cell decomposition. So there are an even number of $4$--cells around each square or triangular face. 

Moreover, $Q_0$ bounds one or two regions $R, R^\prime$ with $2$--dimensional spines, depending on whether $Q$ is one-sided or two-sided.  So this is the $5$--dimensional equivalent of a one- or two-sided Heegaard splitting. Moreover the inclusion maps $\pi_1(Q_0) \to \pi_1(R)$ and $\pi_1(Q_0) \to \pi_1(R^\prime)$ are injections in both the one- and two-sided cases. (Here we are deleting any simplices for which the link is not a sphere to obtain an open manifold.) The reason is that if a loop in $Q_0$ bounds a disc in $R$ or $R^\prime$, then the disc can be pushed off the spine by transversality and hence into $Q_0$. So if $Q_0$ bounds two regions and neither $\pi_1(Q_0) \to \pi_1(R)$ nor $\pi_1(Q_0) \to \pi_1(R^\prime)$ are onto, then $\pi_1(M)$ is an amalgamated free product. \end{example}

\begin{example}
The observations in Example~\ref{exa:even 5} extend in an interesting way to all odd dimensional pseudo-manifolds. In particular if $M$ is a ${(2k-1)}$--dimensional pseudo-manifold, then passing to the covering space corresponding to the kernel of $\rho_{k|k}$, each component $Q_0$ of the hypersurface $Q$ is embedded and one- or two-sided, bounding one or two regions $R, R^\prime$ with $k$--dimensional spines. Consequently the inclusion map induces injections $\pi_j(Q_0) \to \pi_j(R)$ and $\pi_j(Q_0) \to \pi_j(R^\prime)$ for $1 \le j \le k-2$.
\end{example}

\section{Closed 3--manifolds with even triangulations}
\label{sec:Properties of 3-manifolds with even triangulations}

Suppose $M$ is a closed 3--manifold. We already know that if there are at most three vertices in an even triangulation of $M$, then the canonical symmetric representation into $\Sym(4)$ has non-trivial image. Considering the quotients of $\Sym(4)$ and its non-trivial subgroups, this implies $H_1(M,\Z_2) \ne 0$ or $H_1(M,\Z_3) \ne 0.$ It turns out that the interplay between the canonical and the induced symmetric representations gives more information when there are just one or two vertices. This is done in \S\S\ref{subsect:prop one vertex}--\ref{subsect:prop two vertex}.

After determining these properties of closed 3--manifolds having even triangulations, we turn to the question of existence of such triangulations on given 3--manifolds. We build explicit even triangulations of $3$-manifolds with small numbers of vertices, which arise naturally as duals to 1-- and 2--sided Heegaard splittings and also in covering spaces.  We then show that if a $3$--manifold has an even triangulation, then it also has one obtained from our constructions. We wish to emphasize that not all even triangulations arise from our constructions---this would be too much to expect---but that they represent the most common cases.

A brief discussion of similar applications to ideal triangulations can be found in Remark~\ref{rem: ideal prop+constr}.


\subsection{Even triangulations with two vertices}
\label{subsect:prop two vertex}

An important construction of closed orientable 3--manifolds is in terms of Heegaard splittings and one-sided Heegaard splittings. Recall that a Heegaard splitting is a representation of the manifold by gluing two homeomorphic handlebodies together along their boundary surfaces, by an orientation-reversing homeomorphism. A one-sided Heegaard splitting (see \cite{rub}) is obtained by gluing a handlebody to itself by an orientation-reversing involution without fixed points on the boundary of the handlebody.

\begin{theorem}\label{thm:closed 2}
Suppose that $M$ is a closed, orientable $3$-manifold, which admits an even triangulation $\tri$ with exactly two vertices. Then there are three possibilities.
\begin{enumerate}
\item $H_1(M,\Z_2) \ne 0$.
\item $H_1(M,\Z_2) = 0$  and there is an epimorphism of $\pi_1(M)$ onto $\Alt(4)$. 
\item Three faces in the triangulation form a spine for the lens space $L(3,1);$ see Figure~\ref{fig:L(3,1) spine}.\newline In particular, $M$ has $L(3,1)$ as a summand in its prime decomposition.
\end{enumerate}
\end{theorem}

\begin{proof}
We first use the symmetric representations $\prho=\rho_{2|2} \co \pi_1(M) \to \Sym(3)$ and $\rho \co \pi_1(M) \to \Sym(4)$ described above to determine under which conditions $H_1(M,\Z_2) \ne 0$. 

Up to isomorphism, there are four possible images of $\prho$ namely $\{1\}$, $C_2,$ $C_3$ and $\Sym(3).$ Similarly, there are nine possible images of $\rho,$ namely the same as for $\prho$ and $K,$ $C_4,$ $D_4,$ $\Alt(4)$ and $\Sym(4)$, where $K$ is the Klein $4$-group, and $D_4$ the dihedral group of order $8$. There is a homomorphism of $\pi_1(M)$ onto $C_2$ in all cases except when the image of $\rho$ is $ \{1\},$ $C_3,$ or  $\Alt(4)$. So we conclude that either $H_1(M,\Z_2) \ne 0$ or there is a mapping of $\pi_1(M)$ onto $\Alt(4)$ except if the image of $\rho$ is $ \{1\}$ or $C_3$. In these two remaining cases the image of $\prho$ is isomorphic with the image of $\rho$ since we exclude conclusions (1) and (2).

If $\rho$ (and hence $\prho$) has trivial image, each component of the $(2,2)$--hypersurface is an embedded surface, meeting each tetrahedron in exactly one quadrilateral disc. Each of these quadrilateral surfaces is a Heegaard splitting which could be one-sided or two-sided. If any one of these is one-sided, then $H_1(M,\Z_2) \ne 0$ (by taking a loop intersecting a one-sided surface in a single point). On the other hand, we claim that not all three surfaces can be two-sided. By our assumption $\tri$ has two vertices $v,$ and $v^\prime$. Notice that the collection of edges disjoint from one of the quadrilateral surfaces must form two wedges of circles based at $v$ and $v^\prime$ respectively, if this quadrilateral surface is two-sided. But if this is true for all three surfaces, we find that the one-skeleton of $\tri$ is disconnected, since there are no edges running from $v$ to $v^\prime$. So we also have $H_1(M,\Z_2) \ne 0$ when the image of $\rho$ is $ \{1\}.$

It remains to consider the case, where $H_1(M,\Z_2) = 0$ and the image of $\rho$ (and hence of $\prho$) is $C_3,$ so the quadrilateral surface $Q$ is connected. 

Assume first that $Q$ is one-sided. We claim that in this case, $\rho$ maps $\pi_1(M)$ onto $\Alt(4)$, contrary to assumption. By assumption, $\rho$ maps $\pi_1(M)$ onto $C_3$. In the associated covering space $\widetilde{M}$ of $M$, the quadrilateral surface $\widetilde{Q}$ has three embedded components, which are all one-sided and permuted by the covering transformation. Now choose a base tetrahedron $\sigma_0$ in $M$ and a lift $\tilde{\sigma}_0$ in $\widetilde{M}.$ Label the corners of $\sigma_0$ and lift the labelling to $\tilde{\sigma}_0.$ For each component of $\widetilde{Q},$ one can choose an orientation reversing loop based at the barycentre of $\tilde{\sigma}_0$ and contained in the dual 1--skeleton of $\widetilde{\tri}.$ Under the symmetric representation of $\pi_1(\widetilde{M}),$ this loop maps to a product of two disjoint transpositions, corresponding to the involution of $\tilde{\sigma}_0$ stabilising the quadrilateral of the chosen component of $\widetilde{Q}$ in $\tilde{\sigma}_0$ and interchanging the quadrilaterals of the two other components. Proposition~\ref{pro:trivial induced implies embedding} implies that the induced symmetric representation of $\pi_1(\widetilde{M})$ is trivial, and hence the image of the symmetric representation of $\pi_1(\widetilde{M})$ is the Klein four group by Proposition~\ref{pro:sym factors for n=3}. But the action of loops in $\widetilde{M}$ on corners of $\tilde{\sigma}_0$ descends to the same action of their images on the corners of $\sigma_0.$ Whence the image of the symmetric representation $\rho$ of $M$ is generated by the Klein four group and a 3--cycle, and hence $\Alt(4).$ This contradicts our hypothesis.

Hence $Q$ must be orientable. Now $\tri$ has two vertices, which we denote by $v, v^\prime$. Suppose there is a tetrahedron $\sigma$ with two vertices identified with $v$ and two with $v^\prime$. Then there are opposite edges of $\sigma$ which are edge loops $e$ based at $v$ and $e^\prime$ based at $v^\prime$. We claim that in this case, either $e$ or $e^\prime$ has oriented intersection number $\pm 2$ with $Q$, depending on how the loops and surface are oriented. This is easy to see, since if we orient the two quadrilaterals in $\sigma$ which both cross $e, e^\prime$ then the signs of the crossings must be the same on one of the two edges and opposite on the other. So $H_1(M,\Z)$ has an element of infinite order generated by $Q,$ and hence $H_1(M,\Z_2) \ne 0;$ a contradiction.

The same argument applies if there is a tetrahedron $\sigma$ with all vertices identified with $v$ or all vertices identified with $v^\prime$. 

In the remaining case, all tetrahedra must have the same labelling with, say, one vertex identified with $v$ and three vertices identified with $v^\prime$. (There cannot also be tetrahedra with three vertices identified with $v$ and one with $v^\prime,$ since then faces of the two types of tetrahedra do not match and the manifold would be disconnected). Hence the triangulation can be viewed as a double-cone on a one-vertex triangulated $2$-complex $X$, which consists of all the triangular faces with vertices identified to $v^\prime$. Moreover, $X$ is a spine for $M$. 

Lift to the $3$-fold covering space $\widetilde{M}$ of $M$ corresponding to the symmetric representation $\rho$ of $\pi_1(M)$ onto $C_3$. In $\widetilde{M}$, the lifted quadrilateral surface $\widetilde{Q}$ has three components, labelled $0,1,2$. Now we can also label the edges of the lifted triangulation $\widetilde{\tri}$ by $0,1,2,$ where an edge is labelled $i$ if it is disjoint from the component $i$ of $\widetilde{Q}$. The action of $C_3$ cyclically permutes the components of $\widetilde{Q},$ and hence the labels of the edges. Denote the covering transformations on $\widetilde{M}$ by $\{1,g,g^2\}$. Without loss of generality, the covering action and edge labellings can be chosen so that $g(0)=1,g(1)=2,g(2)=0$, i.e.\thinspace $g$ acts as the 3--cycle $(012).$

Each component of $\widetilde{Q}$ is a surface homeomorphic to $Q$ and hence is orientable. Since $X$ is a spine for $M$, $\pi_1(X)$ is isomorphic with $\pi_1(M)$ and so maps onto $C_3$. Hence we can find a generator, which is an edge loop $e$ in $X$, which maps to a non-trivial element of $C_3$. Now $e$ lifts to edges joining different vertices in $\widetilde{X}$, the lift of $X$ to $\widetilde{M}$. Choose a lift $\tilde{e}$ of $e.$ Without loss of generality, we may assume that $\tilde{e}$ has label $0.$ The endpoints of $\tilde{e}$ are distinct lifts of $v'.$ We can denote one of them $\tilde{v}^\prime$ such that the other is $g\tilde{v}^\prime.$ Next, choose a face $\tau$ in $M$ containing $e,$ lift it to a face $\tilde{\tau}$ in $\widetilde{M}.$ Then the remaining vertex of $\tilde{\tau}$ is a lift, denoted $\tilde{v},$ of $v.$

Now the triangulation of $\widetilde{M}$ has the six vertices $\tilde{v},$  $g\tilde{v},$ $g^2\tilde{v},$ $\tilde{v}^\prime,$ $g\tilde{v}^\prime,$ and  $g^2\tilde{v}^\prime.$ Consider a tetrahedron $\tilde{\sigma}$ in $\widetilde{M}$ containing $\tilde{\tau}.$ This tetrahedron has vertices at $\tilde{v},$ $\tilde{v}^\prime,$ $g\tilde{v}^\prime$ and last vertex $\tilde{u}\in \{\tilde{v}^\prime, g\tilde{v}^\prime, g^2\tilde{v}^\prime\}.$ We first show that $\tilde{u}=g^2\tilde{v}^\prime.$

If $\tilde{u}\neq g^2\tilde{v}^\prime,$ then it follows from Proposition~\ref{pro:few vert give non-trivial} that the image of the canonical symmetric representation of $\widetilde{M}$ is non-trivial. Since the induced symmetric representation of $\widetilde{M}$ is trivial, this implies that there is a loop $\tilde{\gamma}$ in $\widetilde{M}$ with image a product of disjoint transposition. Whence the image of the canonical symmetric representation of ${M}$ contains this permutation, contradicting the fact that its image is isomorphic with $C_3.$ This shows that $\tilde{u}=g^2\tilde{v}^\prime.$

The pairs of opposite edges in $\tilde{\sigma}$ are labelled $0, 1, 2.$ We already know that $\tilde{e}=[\tilde{v}^\prime, g\tilde{v}^\prime]_{\tilde{\tau}}$ has label $0.$ Our strategy is to show that the graph $\Gamma_0$ consisting of all edges labelled 0 is connected. This implies that the component of $\widetilde{Q}$ dual to these edges is a 1--sided Heegaard splitting surface, giving the desired contradiction. Denote $C$ the connected component of $\Gamma_0$ containing $\tilde{e}.$ It therefore suffices to show that $\Gamma_0=C,$ and for this equality, it suffices to show that all six vertices are in $C.$

We distinguish two cases, depending on the labels of the remaining edges of $\tilde{\tau}.$

\begin{itemize}
\item[(1)] Suppose that the remaining edges of $\tilde{\tau}$ have the following labels: $[\tilde{v}, \tilde{v}^\prime]_{\tilde{\tau}}$ has label 2 and $[\tilde{v}, g\tilde{v}^\prime]_{\tilde{\tau}}$ has label 1. Since $g^2(1)=0,$ applying $g^2$ to $[\tilde{v}, g\tilde{v}^\prime]_{\tilde{\tau}}$ gives an edge labelled $0$ running from $g^2\tilde{v}$ to $\tilde{v}^\prime.$ Similarly, since $g(2)=0,$ applying $g$ to $[\tilde{v}, \tilde{v}^\prime]_{\tilde{\tau}}$ gives an edge labelled $0$ running from $g\tilde{v}$ to $g\tilde{v}^\prime.$ In particular, $C$ of $\Gamma_0$ contains $g\tilde{v}, g^2\tilde{v}, \tilde{v}^\prime$ and $g\tilde{v}^\prime.$

Since $\tilde{u}=g^2\tilde{v}^\prime,$ then due to the edge in $\tilde{\sigma}$ opposite $\tilde{e}$ we know that $g^2\tilde{v}^\prime$ and $\tilde{v}$ are in the same connected component of $\Gamma_0,$ and it remains to show that one of these vertices is in $C.$ Now there is an edge of $\tilde{\sigma}$ with label 2 and endpoints $g\tilde{v}^\prime$ and $g^2\tilde{v}^\prime.$ Applying $g$ shows that $g^2\tilde{v}^\prime$ is in $C,$ and hence $C$ contains all vertices.

\item[(2)] Suppose that the remaining edges of $\tilde{\tau}$ have the following labels: $[\tilde{v}, \tilde{v}^\prime]_{\tilde{\tau}}$ has label 1 and $[\tilde{v}, g\tilde{v}^\prime]_{\tilde{\tau}}$ has label 2. Applying $g$ to $[\tilde{v}, g\tilde{v}^\prime]_{\tilde{\tau}}$ gives an edge labelled $0$ running from $g\tilde{v}$ to $g^2\tilde{v}^\prime.$ Similarly, applying $g^2$ to $[\tilde{v}, \tilde{v}^\prime]_{\tilde{\tau}}$ gives an edge labelled $0$ running from $g^2\tilde{v}$ to $g^2\tilde{v}^\prime.$ In particular, there is a connected component $C^\prime$ of $\Gamma_0$ containing the three vertices 
$g\tilde{v}, g^2\tilde{v}$ and $g^2\tilde{v}^\prime.$ 
We already know that $C$ contains $\tilde{v}^\prime$ and $g\tilde{v}^\prime.$
It therefore suffices to show that $\tilde{v}$ is in $C$ and $C=C'.$ We again consider cases depending on the remaining vertex of $\tilde{\sigma}.$

Since $\tilde{u}=g^2\tilde{v}^\prime$  and case (1) above leads to a contradiction, we know that either $C = \Gamma_0$ or \emph{every} tetrahedron in $\widetilde{M}$  must have four vertices labelled $g^k \tilde{v}, \tilde{v}^\prime, g\tilde{v}^\prime, g^2\tilde{v}^\prime$, for $k=0,1,2.$ The method used so far merely leads to the conclusion that $\Gamma_0$ has at most two components; one with set of vertices $\{\tilde{v}^\prime, g\tilde{v}^\prime\},$ and the other with set of vertices $\{\tilde{v}, g\tilde{v}, g^2\tilde{v}, g^2\tilde{v}^\prime\}.$ Denote $Y$ the union of all tetrahedra in $\widetilde{M}$ that are incident with $\tilde{v}.$ Then $g^kY$ is the union of all tetrahedra in $\widetilde{M}$ that are incident with $g^k\tilde{v}$ for $k=0,1,2.$ Since $\widetilde{M}$ is connected, there is at least one edge, $\tilde{f},$ in the intersection $Y \cap gY\cap g^2Y.$ Whence the orbit of $\tilde{f}$ is also contained in the intersection. Notice that the union $\tilde{f} \cup g\tilde{f} \cup g^2\tilde{f}$ is an embedded circle in $\widetilde{M}$ passing through the vertices $ \tilde{v}^\prime, g\tilde{v}^\prime, g^2\tilde{v}^\prime.$

Now $Y$ can be viewed as a cone on a 2--complex (possibly with some boundary faces identified). The cone structure of $Y$ implies that there is a disc in the 2--skeleton of $\tri$ which is a cone on $\tilde{f} \cup g\tilde{f} \cup g^2\tilde{f}$ with cone point $\tilde{v}.$ This disc is clearly embedded in $\widetilde{M}$ and the covering map maps the interior of the disc injectively to $M.$ Whence the closed disc is mapped to a spine for $L(3,1)$ in $M.$ This is the third conclusion of the theorem.
\end{itemize}
This concludes the proof of the theorem.
\end{proof}

\begin{example}[(Binary tetrahedral space)]\label{exa:tetrahedral}
A nice example of a triangulation satisfying the second possibility in the theorem is the binary tetrahedral space $S^3/T$, where $T$ is the binary tetrahedral group of order $24$. A fundamental domain for the action of $T$ on $S^3$ is a regular spherical octahedron with all dihedral angles $\frac{2\pi}{3}$. Opposite faces are identified using a $\frac{2\pi}{3}$ twist. If we cone the faces of the octahedron to a vertex in its centre, we obtain a triangulation of the octahedron which glues up to a two vertex triangulation of $S^3/T$  where all edges are of even order. There are $8$ tetrahedra and hence $24$ quadrilaterals in the quadrilateral surface. Each surface has two vertices of degree $4$ and two of degree $6$. Hence each contributes $-\frac{1}{6}$ to the Euler characteristic which is therefore $-4$. The symmetric representation to $C_3$ gives a $3$-fold covering space with  fundamental group $Q_8$, the unit quaternions. (See \cite{rub1} for properties of this covering space.) 
The quadrilateral surface lifts to three Heegaard splittings which are either one-or two-sided.  The argument in Theorem~\ref{thm:closed 2} shows that these splittings must be one-sided. It is not difficult to check that the three lifted quadrilateral surfaces each compress to a (different) Klein bottle one-sided splitting for $S^3/Q_8$.

\end{example}

\begin{example}[(Lens spaces and connected sums)]\label{exa:lens}

We will describe infinitely many even triangulations with two vertices of the lens space $L(3,1)$ and connected sums of this lens space with other 3--manifolds.  These correspond to the third possibility of the theorem. We begin with the simple two tetrahedron triangulation $\tri$ of $L(3,1)$ shown in Figure~\ref{fig:L(3,1)}. To build $\tri$ glue together two tetrahedra along three faces of each, giving a cone on a 2--sphere which is the double of a triangle . Denote the interior vertex of the 3--ball by $v$. Now fold the 2--sphere onto the spine of a lens space, i.e.\thinspace the result of identifying all the edges of a triangle together with the same orientation induced from an orientation of the triangle. It is easy to see that the resulting triangulation of $L(3,1)$ has four edges of degrees $2,2,2,6$ and two vertices. (This also comes from the classical description of lens spaces formed by a double cone where the top set of faces is identified with the bottom set by rotations followed by reflections in the equator of the double cone.) 

\begin{figure}[h!]
\centering

  \begin{center}
 \subfigure[Triangulation]{\includegraphics[scale=0.35]{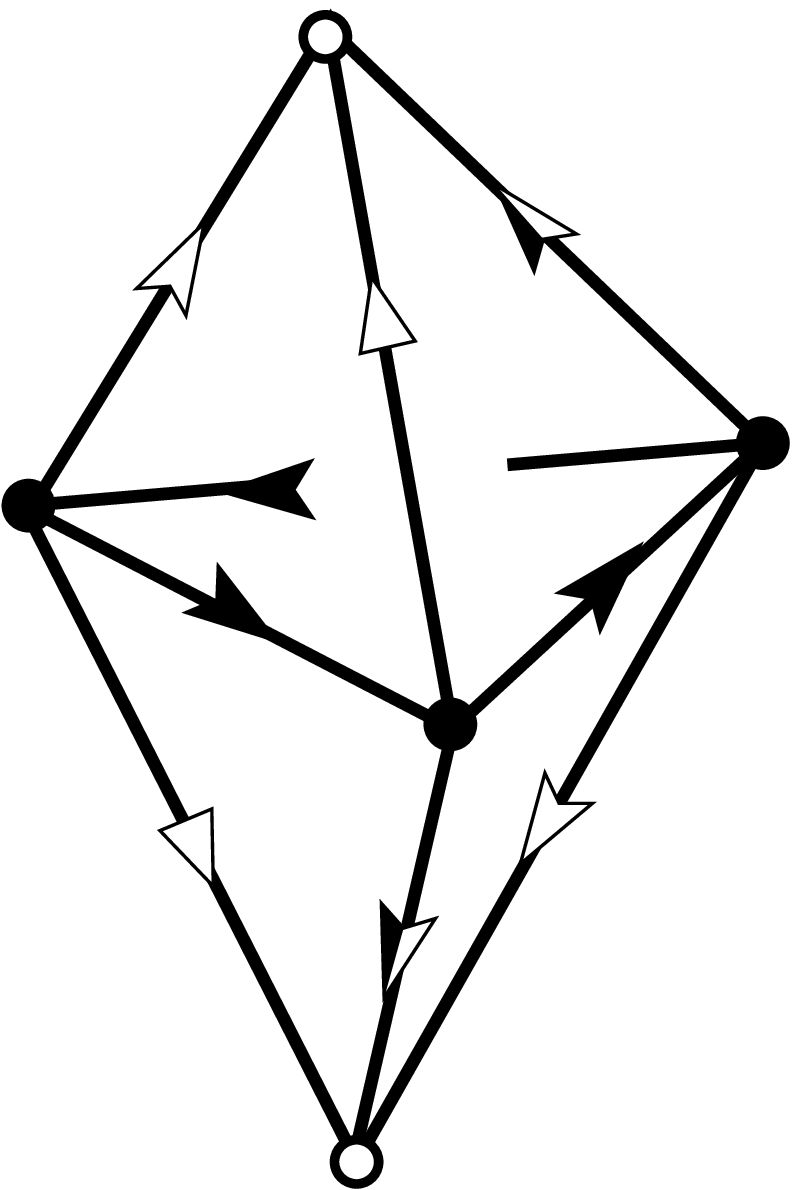}\label{fig:L(3,1)}}
    \qquad\qquad\qquad\qquad
     \subfigure[Spine]{\includegraphics[scale=0.35]{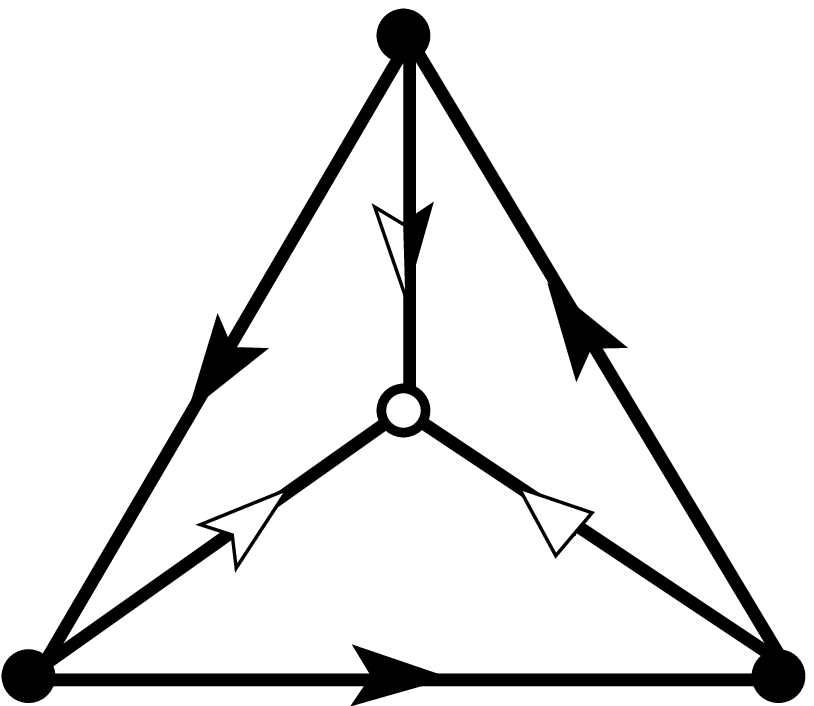}\label{fig:L(3,1) spine}}
\end{center}
\caption{The 2-vertex 2-tetrahedron even triangulation of $L(3,1)$ and the three-face spine}
\end{figure}

We can build on this example in several ways. Let $v^\prime$ denote the second vertex. Firstly, notice that any two distinct edges from $v$ to $v^\prime$ form an embedded loop isotopic to a core circle of the lens space, i.e. to the edge loop based at $v^\prime$. So any cyclic branched cover of degree $n$ over this loop will produce an even triangulation with two vertices of $L(3,1)$ with $2n$ tetrahedra. 

Secondly, suppose we ``split open" the triangulation $\tri$ along a triangular face bounded by two edges from $v$ to $v^\prime$ and the edge loop at $v^\prime$. This gives a triangulation of $L(3,1)$ which has had an open 3--ball removed but has also been pinched so that two vertices of the boundary 2--sphere are glued together. We can now glue on any 3--vertex even triangulation of a closed 3--manifold $M$, which has similarly been split open along a triangular face containing all three vertices. The result is a 2--vertex even triangulation of a connected sum of $L(3,1)$ with $M$.
\end{example}


\subsection{Even triangulations with one vertex}
\label{subsect:prop one vertex}

\begin{theorem}\label{thm:closed 1}

Suppose that $M$ is a closed, orientable $3$-manifold, which admits an even triangulation $\tri$ with exactly one vertex. Then there is an epimorphism of $\pi_1(M)$ onto $C_2\times C_2,$ $C_4,$ $\Alt(4)$ or $\Sym(4).$
\end{theorem}

\begin{proof} 

As for Theorem~\ref{thm:closed 2}, we start with the symmetric representations $\rho$ and $\prho$. The image of $\prho$ is one of $\Sym(3),$ $C_3,$ $C_2,$ $\{1\}$. Notice that if the quadrilateral surface $Q$ has three components, as occurs if the image of $\prho$ is the trivial subgroup, then each of these is a one-sided Heegaard splitting. The reason is that the spine of the complement of each component must be connected, since there is only one vertex. So in this case, we claim that $H_1(M,\Z_2)$ has rank at least $2$. To verify the claim, note the labelling scheme for the quadrilaterals applied to the edges gives edge loops with distinct labels dual to the three components of $Q$. So the classes of the edge loops give three different elements of first homology with $\Z_2$ coefficients. 

If the induced symmetric representation $\prho$ has image $C_2$, then $\rho$ has image $C_2,$ $C_2\times C_2,$ $C_4$ or $D_4.$ The conclusion of the theorem hold unless the image of $\rho$ is $C_2.$ Since $\prho$ also has image $C_2,$ the image of $\rho$ is generated by a 2--cycle and there are two components of $Q$ and one of these is embedded. As in the previous paragraph, this surface must be a one-sided Heegaard splitting $H$ for $M$. The double cover $\widetilde{M}$ associated to $\prho$ is the same as the double cover corresponding to $H.$ Since $H$ lifts to a 2-sided Heegaard surface in $\widetilde{M}$ and meets each tetrahedron in exactly one quadrilateral, all tetrahedra in the cover have two corners at one of the vertices, $\tilde{v}$ of the cover, and two corners at the other vertex, $g\tilde{v},$ where $g$ is a deck transformation, and the pairs are separated by $\widetilde{H}.$
The quadrilateral surface $\widetilde{Q}$ in the cover has three embedded components, and the other
two components are 1--sided Heegaard surfaces since their complementary spine contains both $\tilde{v}$ and $g\tilde{v}$ and hence is connected. These components are permuted by the covering transformation. An orientation reversing loop on either of them is mapped to a product of transpositions under the symmetric representation of the cover, and hence we get the same product of transposition by applying $\rho$ to the image of this loop in $\pi_1(M).$ But this involution together with the 2--cycle Shows that the image of $\prho$ contains a group isomorphic with $C_2\times C_2$ or $D_4;$ a contradiction.

Finally if the induced symmetric representation $\prho$ has image $\Sym(3)$ or $C_3$, then consider the associated covering space $\widetilde{M}$ of $M$. The quadrilateral surface separates into three components in $\widetilde{M}$. If these are one-sided, then $H_1(\widetilde{M},\Z_2)$ has an associated subgroup of rank $2$ and  we get an associated representation $\rho$ of $\pi_1(M)$ onto $\Sym(4)$ or $\Alt(4)$ respectively, as in the proof of Theorem~\ref{thm:closed 2}.  Hence assume that the components are orientable. We proceed as in the argument for Theorem~\ref{thm:closed 2} to analyse cases, noting that the lifted triangulation $\widetilde{\tri}$ has either six or three vertices respectively. 

First suppose $\prho$ has image $C_3,$ so $\widetilde{\tri}$ has three vertices. We denote these vertices by $\hat  v, g\hat  v, g^2 \hat  v$ where $\{1,g,g^2\}$ denotes the covering transformations. Assume that a component of the quadrilateral surface is a two-sided Heegaard splitting $\Sigma$ for $\widetilde{M}$. Then without loss of generality we can assume that $\hat  v, g\hat  v$ are the vertices belonging to one of the spines for $\Sigma$ and $g^2 \hat  v$ is the vertex belonging to the other one. Also suppose that the labelling of the edges in these spines is $0$ and that the covering transformation acts as $(012)$ on the labels. Suppose there is a tetrahedron $\tilde{\sigma}_0$ having a corner at $\tilde v$ and a corner at $g^2\tilde v.$ Then there must be another corner at $g^2\tilde v,$ and hence 
$\tilde v$ and $g^2\tilde v$ are connected by both an edge labelled 1 and an edge labelled 2 in $\tilde{\sigma}_0.$ But applying $g^2$ to the former gives an edge labelled 0 connecting $g^2\tilde v$ and $g\tilde v,$ contradicting the fact that they lie in disjoint spines. Hence, since the manifold is connected, there is a tetrahedron $\tilde{\sigma}_1$ having a corner at $g\tilde v$ and a corner at $g^2\tilde v.$ Again, there is another corner at $g^2\tilde v,$ and we get edges labelled 1 and 2 respectively joining $g\tilde v$ to $g^2\tilde v.$ Now applying $g$ to the latter gives a contradiction.

In the last case, where the image of the induced symmetric representation $\prho$ is $\Sym(3)$, we can apply Theorem~\ref{thm:closed 2}. Namely first pass to the double covering $\widetilde{M}$ corresponding to the representation $\pi_1(M) \to \Sym(3) \to C_2$. Then the triangulation of $M$ lifts to an even triangulation of $\widetilde{M}$ with two vertices and a symmetric representation with image $C_3$. So we can apply Theorem~\ref{thm:closed 2} to conclude that in this case, either $H_1(\widetilde{M}, \Z_2) \ne 0$ or there is a representation of $\pi_1(\widetilde{M})$ onto $\Alt(4)$ or there is a connect summand which is a copy of $L(3,1)$. In the first two cases it is straightforward, as in the above argument for $C_2,$ to deduce from the way that the representation arises that there is an epimorphism of $\pi_1(M)$ onto one of $C_2 \times C_2,$ $C_4,$ $\Alt(4),$ or $\Sym(4).$ 

If there is a connected summand which is a copy of $L(3,1),$ then the argument in case (2c) in the proof of Theorem~\ref{thm:closed 2} shows that all tetrahedra in the 2--vertex triangulation of $\widetilde{M}$  have one corner at the vertex $\tilde{v}$ and three corners at the vertex $\tilde{v}^\prime.$ But the covering transformation interchanges the two vertices, and hence would take such a tetrahedron to one with one corner at the vertex $\tilde{v}^\prime$ and three corners at the vertex $\tilde{v}.$ So this case cannot occur. 
\end{proof}

\begin{example}[(Klein bottle lens spaces and prism manifolds)]
Consider the lens space $L(4,1)$ with the 1--vertex triangulation $\tri$ dual to the one-sided Heegaard splitting given by a Klein bottle $K$. This means a solid torus $Y$ is attached to $K$ to form the lens space (see \cite{rub1}). To describe the triangulation, if we write $\pi_1(K) = \{a,b:b^{-1}ab=a^{-1}\}$ then the meridian disc for the solid torus $Y$ has slope $b^2a$. The union of $K$ and the disc form a spine and the triangulation is dual to this spine. The disc has even intersection number with itself (it has a single self-intersection which is counted twice) and so the triangulation $\tri$ has a single tetrahedron, one vertex and all edges of even order. In fact, this is the triangulation described in Example~\ref{exa:L(4,1)}. In this case, $\pi_1(L(4,1))=C_4$ so this illustrates the second case of Theorem 22. All the Klein bottle lens spaces $L(4k,2k-1)$ and prism manifolds, i.e.\thinspace the manifolds with Klein bottles and dihedral or binary dihedral by cyclic finite fundamental groups have similar 1--vertex even triangulations for the same reasons. For the dihedral case, there is a homomorphism of the fundamental group to $C_4$ and for the binary dihedral case to $C_2 \times C_2$. 
See \cite{rub1} for more information about these manifolds and their one-sided splittings. 
\end{example}

\begin{example}

In \cite{bir} a number of examples of $3$-manifolds with one-sided Heegaard splittings of cross-cap genus $3$ are given. In particular, the class of Seifert fibered spaces 
$$\{\;b;\;(o_,0); \; (2,1),\;(4k,2k-1),\;(m,n)\;\}$$ 
for $1 \le k$ and $0<n<m$ or $(m,n)=(1,0)$ are discussed and explicit one-sided Heegaard diagrams drawn. (Here we are using the notation of \cite{Or}). Some of these examples have finite fundamental group (cyclic or binary octahedral by cyclic), but most have infinite fundamental group. Since the Heegaard diagrams have only two discs, a necessary and sufficient condition for all edges to be even in the dual triangulation is that the two discs have even intersection number. From the representation of the homology classes of these discs in \cite{bir} one can compute that this occurs if and only if $mk$ is even. In this case, $\pi_1(M)$ maps onto $C_2 \times C_2$ or $C_4$ depending on whether $m$ or $k$ is even respectively. 
\end{example}

\begin{problem}
If $M$ has an even 1--vertex triangulation, is $H_1(M,\Z_2)$ necessarily non-trivial? We know this fails for 2--vertex triangulations, from Examples~\ref{exa:tetrahedral} and \ref{exa:lens}, and the minimal 1--vertex triangulations of both the binary tetrahedral space and the lens space $L(3,1)$ have some edges of odd degree. 
\end{problem}


\subsection{Existence of even triangulations with two vertices}
\label{subsec:Existence of even triangulations with two vertices}

A $3$-manifold is \emph{irreducible}, if every embedded 2--sphere bounds a 3--ball. A Heegaard splitting is called \emph{irreducible} if whenever a 2--sphere meets the Heegaard surface in a single loop, then the loop bounds a disc in the Heegaard surface. 

\begin{theorem}\label{thm:construct 2-vertex}
Suppose $M$ is a closed irreducible 3--manifold with a diagram for an irreducible Heegaard splitting, satisfying every disc in the diagram has even total intersection number with the other discs in the diagram. Then $M$ admits an even triangulation dual to the Heegaard diagram, which has exactly two vertices. Moreover if $M \ne \R P^3$, then the dual triangulation can be chosen to have no edges of degree $2$.
\end{theorem}

\begin{proof}
Start with an irreducible Heegaard splitting $\Sigma$ of a closed irreducible 3--manifold $M$. We will assume for simplicity that $M$ is orientable---the non-orientable case is easier since then the homological condition we require, is satisfied automatically. Let the two handlebodies for this splitting be denoted $Y, Y^\prime$. Assume that $H_1(M,\Z_2) \ne 0$. We want to pick a Heegaard diagram $\di$ for this splitting with the following property:

Let the diagram $\di$ consist of a complete systems of meridian discs for $Y$ (respectively $Y^\prime$), given by $D_1, D_2, \dots D_k$ (respectively $D^\prime_1, D^\prime_2, \dots D^\prime_k$), where $k$ is the genus of $\Sigma$,
so that the total number of intersections of the boundary of any disc in $\di$ with all the boundaries of the other system of discs is even. 

We claim that the dual triangulation $\tri$ to this Heegaard diagram (i.e. to the spine $\Sigma \cup \di$) has the properties that all its edges are of even order and there are precisely two vertices. 

Since the result of cutting open handlebody $Y$ along the meridian discs $D^\prime_1, D^\prime_2, \dots D^\prime_k$  is a ball, we get one vertex $v$ dual to this ball and similarly a second vertex $v^\prime$ dual to the ball similarly obtained by cutting open the handlebody $Y^\prime$ along its meridian discs. So there are exactly two vertices in $\tri$.

Secondly, there are three types of edges in $\tri$ - dual to discs $D_i$, dual to discs $D^\prime_j$ and dual to faces of the diagram $\di$ on $\Sigma$. 
The first two types of edges are of even order, since their order is the same as the number of intersection points of the boundaries of the discs $D_i$, and $D^\prime_j$ with the other discs of $\di$. By assumption, these numbers are all even. Finally faces of the Heegaard diagram on $\Sigma$ must be of even order, since arcs of the faces alternate between meridian discs in the two systems. By choosing the discs to have minimal intersection, we can assume that there are no faces of $\Sigma$ of degree two and all faces are discs, by the assumption that the splitting is irreducible and the manifold is irreducible. Finally, if $M \ne \R P^3$, then no edge of the first two types can be of degree two either. The reason is that a degree two edge would correspond to a disc say $D_i$ with boundary meeting all  the discs $D^\prime_j$ in a total of two points. Now if there are two such discs each met once, then there would be a trivial handle of the splitting, contradicting our assumption that it is irreducible. If there are two intersections with a single disc $D^\prime_j$, then it is easy to see that there is an embedded copy of $\R P^2$ obtained by attaching a Mobius band in $Y^\prime$ to $D_i$. But then since $M$ is irreducible, we see that it must be $\R P^3$. 
\end{proof}

\begin{remark}
Notice that if we take the usual Heegaard diagram for $\R P^3$ using a Heegaard torus, then the dual triangulation does have two vertices and all edges of even order, but also there are edges of degree two. 
\end{remark}

\begin{remark}
If we take any irreducible Heegaard splitting for any irreducible closed orientable 3--manifold and take two parallel copies of a complete system of meridian discs for each handlebody, we get a triangulation with all edges having even order and $2k+2$ vertices, where $k$ is the genus of the Heegaard splitting. This again shows that to have an interesting theory, we need to restrict to small numbers of vertices. The case of $1$--vertex triangulations is particularly interesting, since minimal triangulations are of this type (see \cite{JR}). 
\end{remark}

\begin{remark}
Another construction is given by a strongly irreducible Heegaard splitting for a closed orientable 3--manifold. Recall that for such a splitting, every meridian disc for one handlebody meets every meridian disc for the other handlebody. In this case, pick two systems of meridian discs for the two handlebodies which each separate the handlebodies into two 3--balls. Then the dual triangulation has four vertices. Moreover every edge is of even order. For edges dual to faces on the Heegaard splitting, the reason is as above. For edges dual to meridian discs, since the splitting is strongly irreducible, each disc for one handlebody meets all the discs for the other handlebody. Hence the boundary of such a disc is divided into an even number of arcs lying in the two 3--balls. Hence the dual edges have even degree as claimed. This shows that even triangulations possessing four vertices are very common. 
\end{remark}

\begin{theorem}\label{thm:closedexist 2}
Suppose that $M$ is a closed, irreducible, orientable 3--manifold. If $H_1(M,\Z_2)$ is non-trivial, then $M$ admits an even triangulation $\tri$ with exactly two vertices. Moreover this triangulation can be chosen as the dual to a Heegaard diagram. If $M \ne \R P^3$ then the triangulation can also be chosen to not have any edges of degree $2$.
\end{theorem}

\begin{proof}
As in the proof of Theorem~\ref{thm:construct 2-vertex}, we begin with an irreducible Heegaard splitting $\Sigma$ for $M$ and a Heegaard diagram. The idea is to modify the Heegaard diagram by disc band sums until we achieve the condition that each disc meets all the other discs in an even total intersection. 

Note firstly that band sums of the discs on either side of $\Sigma$ correspond to row or column operations on the square matrix of geometric intersection numbers of the discs. We can use elementary algebra to reduce the matrix to a diagonal form, where the diagonal elements each are divisors of the next ones down the main diagonal. For instance, the first diagonal element is obtained by using Euclid's algorithm for the greatest common divisor of all the matrix entries, then the next diagonal element is the greatest common divisor of the submatrix obtained by deleting the first row and column, and then one keeps iterating this procedure.

Since $H_1(M,\Z_2) \ne 0$ we see that the product of the diagonal entries must be even. Hence the last diagonal entry is even. If we add all the rows (except the last one) to the last row, it is easy to see the result is a matrix where all the column sums are even. Then we can add all the columns (except the last column) to the last column and get in addition that all the row sums are even. This completes the main result.

Exactly as in the previous section, if $M \ne \R P^3$, since the splitting is irreducible we can arrange that no edge has degree $2$.  
\end{proof}

\begin{problem}
Suppose $M$ is a closed, irreducible, orientable $3$--manifold and there is a representation of $\pi_1(M)$ onto $\Alt(4)$. Is there always a 2--vertex triangulation $\tri$ of $M$ so that all the edges have even order and all the tetrahedra contain a single copy of $v$ and three copies of $v^\prime$, where $v,v^\prime$ are the two vertices of $\tri$? It was noted in Example~\ref{exa:tetrahedral} that $S^3/T$, where $T$ is the binary tetrahedral group, satisfies $T$ maps onto $\Alt(4)$ and has such a triangulation. There is a family of generalised binary tetrahedral groups of order $8 \cdot 3^k$ for $k=1,2,3 \dots$ which are fundamental groups of elliptic $3$--manifolds and all have representations onto $\Alt(4)$. (See \cite{Or} for more details.) Do these $3$--manifolds also have even 2--vertex triangulations of this type?

\end{problem}


\subsection{Existence of even triangulations with one vertex}

\begin{theorem}\label{thm:heeg 1}
Suppose $M$ is a closed, irreducible, orientable $3$--manifold with an irreducible 1--sided Heegaard diagram so that every disc in the diagram has even total intersection number with the other discs in the diagram. Then $M$ admits an even triangulation dual to the Heegaard diagram, which has exactly one vertex. Moreover if $M \ne L(4,1)$, then the dual triangulation can be chosen to have no edges of degree $2$.
\end{theorem}

\begin{proof}
This construction is very similar to the previous one of Theorem~\ref{thm:construct 2-vertex}, so we just give a summary. We start with a 1--sided Heegaard diagram $\di$ for a closed orientable irreducible $3$-manifold $M$. (See \cite{rub}). So there is an embedded 1--sided surface $K$ in $M$ so that the closure of the complement $M \setminus K$ is a handlebody $Y$. Therefore $\partial Y$ is the orientable double covering $\widetilde{K}$ of $K$. As before, the Heegaard diagram is a complete set of meridian discs $D_1, D_2, \dots D_k$ for $Y$ where $k$ is the genus of $\partial Y$. The triangulation $\tri$ we are interested in is the dual to the Heegaard diagram, i.e the spine $K \cup \di$. Note that now the complement of this spine is a single ball and so $\tri$ is a $1$--vertex triangulation. 

Next, the edges of $\tri$ are either dual to meridian discs $D_i$ or to faces of the Heegaard diagram on $K$. For the latter it is again easy to see that these are of even order. For if we lift to the canonical double covering $\widetilde{M}$ of $M$ so that $Y$ lifts to two handlebodies and $K$ lifts to its orientable double covering $\widetilde{K}$, we see that the faces in $K$ lift to faces in $\widetilde{K}$. The latter are of even order as before. Again we need to assume that we have picked an irreducible 1--sided Heegaard splitting to ensure that all these faces are discs. Moreover none of the faces will have order two if the discs are arranged to have minimal intersection and self intersection number. 

Finally we see that if each disc $D_i$ has even total intersection with all the other discs $D_j$ for $j \ne i$, then the dual edge to $D_i$ will have even order. The only thing we need to take care of is self intersections of $\partial D_i$. But these are counted twice when computing the degree of the edge dual to $D_i$ so we can ignore them. Finally no edge dual to a disc can have degree two if the Heegaard splitting is irreducible. For if two discs meet once, then there is a cancelling pair (see \cite{br}). On the other hand, if two discs meet twice, then again we are in the situation of having an embedded $\R P^2$ in the double covering $\widetilde{M}$. Hence $M = L(4,1)$. 
\end{proof}

\begin{remark}
The manifold $\R P^3$ is again a special case. Its unique irreducible 1--sided Heegaard diagram has genus zero and so is not dual to a triangulation at all. See \cite{br}.
\end{remark}

\begin{theorem}\label{thm:closedexist 1}
Suppose that $M$ is a closed, irreducible, orientable $3$-manifold so that $\pi_1(M)$ maps onto either $C_2 \times C_2$ or $C_4$. Then $M$ admits an even triangulation with exactly one vertex. Moreover this triangulation is dual to a 1--sided Heegaard diagram. 
\end{theorem}

\begin{proof}
The idea is very similar to Theorem~\ref{thm:closedexist 2}. Namely, since $H_1(M,\Z_2) \ne 0$ we can build a 1--sided Heegaard splitting $K$ for $M$. The approach is then to do disc band sums on the meridian discs for a disc diagram $\di$ to achieve that all discs have even total intersection number with the other discs of the system. The main difficulty is that we will have to work in the double covering $\widetilde{M}$ and do disc moves equivariantly. It turns out that these are just the same as doing simultaneous row and column operations on a symmetric matrix. 

So as before, denote the covering involution by $g$ and assume that the handlebodies for the 2--sided Heegaard splitting $\widetilde{K}$ are denoted $\widetilde{Y}, g(\widetilde{Y})$. Note that $\di$ and $\tri$ lift to $\tilde \di$ and $\widetilde{\tri}$. The former consists of two families of discs $\widetilde{D}_1, \widetilde{D}_2, \dots \widetilde{D}_k$ for $\widetilde{Y}$ and $g(\widetilde{D}_1), g(\widetilde{D}_2), \dots g(\widetilde{D}_k)$ for $g(\widetilde{Y})$. Clearly the geometric intersection matrix for this disc system is symmetric, with the action of $g$ interchanging rows and columns. 

Now by our topological assumptions, we know that  $H_1(\widetilde{M}, \Z_2) \ne 0$. Hence we know that the intersection matrix has zero determinant working in $\Z_2$. Our aim is to do row and column operations simultaneously to convert the matrix to have all row sums (and therefore also columns sums) even. Then the dual triangulation will be $g$--equivariant, with two vertices and all edges of even degree. The triangulation will project to a triangulation of $M$ with the required properties. 

Consider the first column of the intersection matrix. If there is a one in this column, we can permute rows to shift this to just below the main diagonal. 
(Recall that the main diagonal entries are all zero, since the corresponding intersection numbers are even). We can now zero out all the entries below so that there is a single one in the first column, using row operations and the corresponding column operations. Otherwise all the entries in this column are zero. By induction it follows that we can arrange that all the entries which are not adjacent to the main diagonal, are zero, by following the same procedure. Since row and column operations do not affect the determinant which is still zero in $\Z_2$. 

Note that each column and each row of our matrix has at most two entries $1$ and all other entries zero. Moreover if the matrix is non-zero, there are at least two columns with a single entry $1$. If the matrix can be decomposed into diagonal blocks, we can clearly work with the blocks individually. Note that at least one such block must have determinant zero and the other blocks could have determinant $1$. So in the case of blocks we have to explain how to deal with the blocks with non-zero determinant. To start with, we assume there is a single block. In particular, this means we can assume there are no columns with all zero entries. 

We divide the argument into cases. For the first case, assume we have a $k \times k$ block with $k \ge 3$ and the first and last column have a single $1$ and all other columns have two $1$ entries. Perform the following column and row operation on the block. Add copies of columns $j$ for $2 \le j \le k-1$ to column $1$ and similarly for rows. It is easy to verify this produces a matrix where each row and column sum is even; in fact there are precisely two entries $1$ in each row and column. 

Next, if we have a block matrix, possibly with some zero blocks, so long as the non-zero blocks are of the form in the previous paragraph, clearly the same argument works. So we are left with the case of some $2 \times 2$ blocks. Since the matrix has non-zero determinant, there must be either some zero blocks or some $k \times k$ blocks with $k \ge 3$. If there is a zero block, then a simple process converts a $2 \times 2$ block into a $3 \times 3$ block. Namely add a column and row of the $2 \times 2$ block to an adjacent zero block. 

So the problem is reduced to a final case, where there is a combination of $k \times k$ blocks and $2 \times 2$ blocks with no zero blocks. 
Consider the case of an adjacent $2 \times 2$ block and a $k \times k$ block. Add a copy of each of the two columns of the $2 \times 2$ block to 
each of the first and last columns of the $k \times k$ block and do similar row sums. It is easy to verify that this converts all the four columns with a single $1$ entry to have two $1$ entries and likewise for the rows. Hence this completes the proof in all cases. 
\end{proof}

As mentioned in the introduction, a key result due to Lubotzky~\cite{Lu} is that any complete hyperbolic $3$--manifold of finite volume has a finite sheeted covering so that the rank of $\Z_2$--homology is arbitrarily large. Theorem~\ref{thm:closedexist 1} therefore implies:

\begin{corollary}
Any closed hyperbolic $3$--manifold of finite volume has a finite sheeted covering with a 1--vertex even triangulation.
\end{corollary}

\begin{example}\label{exa:minimal counter}
The Seifert fibered space $M = S^2(\;(2,1)\;(2,1)\;(2,1)\;)$ in the ``Closed Orientable Census" of Regina~\cite{Regina} has a unique minimal triangulation with four tetrahedra and the degree sequence of the edges is $4, 5, 5, 5, 5.$ Now $H_1(M)=\Z_2 \oplus \Z_6,$ and hence $\pi_1(M)$ maps onto $C_2 \times C_2.$ Whence $M$ has an even triangulation. An even triangulation with six tetrahedra was found by applying Pachner moves to the minimal triangulation; the degree sequence of the edges is $4, 4, 4, 4, 4, 4, 12.$ The edge of degree 12 meets each tetrahedron in a pair of opposite edges. The quadrilateral surface has two components: an embedded Klein bottle and an immersed surface of Euler characteristic $-4.$ The Klein bottle meets each tetrahedron in the quadrilateral disc disjoint from the degree 12 edge and is a 1--sided Heegaard splitting surface for $M.$
\end{example}

\begin{problem}
Suppose $M$ is a closed, orientable, irreducible $3$--manifold and there is a representation of $\pi_1(M)$ onto $\Alt(4)$ or $\Sym(4)$. Is there an even triangulation of $M$ with exactly one vertex? The binary tetrahedral and binary octahedral spaces are interesting examples to resolve.
\end{problem}

\begin{remark}[(Ideal triangulations)]\label{rem: ideal prop+constr}
Many of the arguments used in this section work when some or all of the vertex links have non-positive Euler characteristic, and hence can be applied to general 3--dimensional pseudo-manifolds $\widehat{M}.$ However, more cases arise as the number of ideal vertices in the cover may be smaller than the covering degree times the number of ideal vertices in $\widehat{M}.$ Moreover, there are analogous constructions of even ideal triangulations using 1--sided or 2--sided splittings into compression bodies. 
\end{remark}


\section{Very short hierarchies}
\label{sec:Very short hierarchies}

Haken $n$--manifolds are defined in \cite{fr}. In particular, a closed orientable Haken $n$--manifold can be cut open along a collection of embedded closed and compact orientable hypersurfaces called a \emph{hierarchy} to a collection of $n$--cells, termed Haken cells. These cells have a boundary pattern given by a decomposition of the boundary of each cell into polyhedra. The polyhedra are the intersections of the hypersurfaces of the hierarchy with the boundary of the Haken cell. 

We are especially interested in when the hierarchy is as `short as possible'. In dimension three, such hierarchies are called \emph{very short}. It is clear that we need to cut along at least $n$ collections of hypersurfaces, so when this number is sufficient, the hierarchy is deemed very short. It is well-known that Haken $3$--manifolds all have very short hierarchies but in higher dimensions this is an open question. 

We have the following interesting connection between very short hierarchies and even triangulations.

A Haken $n$--cell coming from a very short hierarchy is an $n$--dimensional polytope with the property that the codimension one faces can be $n$--coloured, so that no two faces sharing a codimension two face have the same colour. In particular, we see that at every vertex of the polytope, the faces take all possible $n$ colours. 

Now consider the dual triangulation to the cell structure of the boundary of an $n$--dimensional polytope with such a colouring scheme. This is a triangulation of $S^{n-1}$ and the vertices are dual to the faces and hence are $n$--coloured so that no two vertices at the ends of an edge have the same colouring. But this is precisely what we get from our labelling scheme in \S\ref{sec:sym reps and normal}. In particular, the triangulation must be even. Since the fundamental group of $S^{n-1}$ is trivial for $n \ge 3$, it follows that the canonical symmetric representation has trivial image. 

Conversely, suppose we start with any even triangulation of $S^{n-1}$. The canonical symmetric representation must have trivial image since the fundamental group is trivial. Hence there is a labelling of the vertices of the type we want, so that any two vertices at the ends of an edge have different labels. So if we consider the dual cell structure, this can be viewed as the boundary of an $n$--polytope where the faces are $n$--coloured. So this shows that Haken $n$--cells coming from very short hierarchies have boundaries which are dual to even triangulations. 

Finally, the Haken condition says that each pair of faces of the polytope meet in a single $(n-2)$--face and if three faces intersect in pairs in $(n-2)$--faces then they meet together in an $(n-3)$--face. Translating this into a condition on the dual triangulation means that no two edges share two vertices and if three edges have vertices in common, then they are the boundary of a triangle. 

So we can form a Haken $n$--manifold with a very short hierarchy by choosing a collection of even triangulations of $S^{n-1}$ with the properties that no two edges share two vertices and if three edges have vertices in common, then they are the boundary of a triangle. We can then $n$--colour the dual cell decompositions and can glue these together in a colour preserving manner. 

\begin{example}
For a simple example, consider gluing up a $3$--cube to form a $3$--torus. The corresponding even triangulation of $S^2$ is clearly an octahedron and the corresponding colouring gives opposite vertices the same colour. The gluing then identifies opposite vertices. 
\end{example}




\address{Department of Mathematics and Statistics, The University of Melbourne, VIC 3010, Australia} 
\address{School of Mathematics and Statistics, The University of Sydney, NSW 2006, Australia} 
\email{rubin@ms.unimelb.edu.au} 
\email{tillmann@maths.usyd.edu.au} 

\Addresses

\end{document}